\documentclass[onefignum,onetabnum]{siamart190516}

\usepackage{psfrag,mathrsfs,color,empheq}
\usepackage{color,bm,enumerate}
\usepackage{epsfig,soul}
\usepackage{multirow,makecell}

\usepackage{epstopdf}
\usepackage{mathtools} % assumes amsmath package installed
\usepackage{amssymb,bbm}

\DeclareMathOperator*{\argmin}{arg\,min}
\usepackage[sort,nocompress]{cite}
\newcommand{\IR}{{\mathbb{R}}}

\newcommand{\opt}{\rm{opt}}

\newcommand{\bmt}{\left[ \begin{array}{ccccccccc}}
	\newcommand{\emt}{\end{array}\right]}
\newcommand{\bean}{\begin{eqnarray*}}
	\newcommand{\eean}{\end{eqnarray*}}
\newcommand{\bea}{\begin{eqnarray}}
	\newcommand{\eea}{\end{eqnarray}}
\newcommand{\eq}{\begin{equation}\begin{array}{lllllllll}}
		\newcommand{\ee}{\end{array}\end{equation}}
\newcommand{\eqn}{\begin{equation*}\begin{array}{lllllllll}}
		\newcommand{\een}{\end{array}\end{equation*}}
\graphicspath{{./Figures/}}

\usepackage{lipsum}
\usepackage{amsfonts}
\usepackage{graphicx}
\usepackage{epstopdf}
\usepackage{algorithmic}
\ifpdf
\DeclareGraphicsExtensions{.eps,.pdf,.png,.jpg}
\else
\DeclareGraphicsExtensions{.eps}
\fi

\newsiamremark{remark}{Remark}
\newsiamremark{hypothesis}{Hypothesis}
\crefname{hypothesis}{Hypothesis}{Hypotheses}
\newsiamthm{claim}{Claim}

\def\om{\omega}
\def\th{\theta}
\def\R{\mathbb R}

\title{Optimized sparse approximate inverse smoothers for\\   solving Laplacian linear systems\thanks{Submitted to the editors \today.
	}}

\author{Yunhui He\thanks{Department of Computer Science, The University of British Columbia, Vancouver, BC, V6T 1Z4, Canada.
		(\email{yunhui.he@ubc.ca}).}
	\and Jun Liu\thanks{Corresponding author. Department of Mathematics and Statistics, Southern Illinois University Edwardsville, Edwardsville, IL 62026, USA.
		(\email{juliu@siue.edu}).}
	\and Xiang-Sheng Wang\thanks{Department of Mathematics, University of Louisiana at Lafayette, Lafayette, LA 70503, USA. (\email{xswang@louisiana.edu}).}
}

\usepackage{amsopn}
\DeclareMathOperator{\diag}{diag}

\begin{document}

\maketitle

% REQUIRED
\begin{abstract}
	In this paper we propose and analyze new efficient sparse approximate inverse (SPAI) smoothers for solving
	the two-dimensional (2D) and three-dimensional (3D) Laplacian linear system with geometric multigrid methods.
	Local Fourier analysis  shows that our proposed SPAI smoother for 2D achieves a much smaller smoothing factor than the state-of-the-art SPAI smoother studied in [Bolten, M., Huckle, T.K. and Kravvaritis, C.D., 2016. Sparse matrix approximations for multigrid methods. Linear Algebra and its Applications, 502, pp.58-76.]. The proposed SPAI smoother for 3D cases provides smaller optimal smoothing factor than that of weighted Jacobi smoother. Numerical results  validate our theoretical conclusions and illustrate  {the high-efficiency and high-effectiveness}   of our proposed SPAI smoothers.
	{Such SPAI smoothers have the advantage of inherent parallelism.}
\end{abstract}

% REQUIRED
\begin{keywords}
  multigrid, local Fourier analysis, sparse approximate inverse, smoothing factor, Laplacian
\end{keywords}

% REQUIRED
\begin{AMS}
   49M25, 49K20, 65N55, 65F10
\end{AMS}
\section{Introduction}
Laplace operator or Laplacian is used ubiquitously in describing various physical phenomena through partial differential equation (PDE) models, such as Poisson equation, diffusion equation, wave equation, and Stokes equation.
For efficient numerical solution of such PDE models, a fundamental task is to (approximately) solve the corresponding sparse discrete Laplacian linear system upon suitable discretization schemes (e.g., finite difference method), where both direct and iterative solvers are extensively studied in the past century.
In this paper, we will focus on using multigrid solver for the two-dimensional (2D) and three-dimensional (3D) Laplacian system with the 5-point and 7-point stencil central finite difference scheme, respectively, where new sparse approximate inverse (SPAI)  smoothers are our major contribution.
More specifically, we consider the following Poisson equation
on a unit square domain $\Omega=(0,1)^d, d=2,3$ with Dirichlet boundary condition:
\eq \label{Poisson}
-\Delta u=f\quad\mbox{in}\quad\Omega,\qquad u=g \quad\mbox{on}\quad\partial\Omega,
\ee
which, upon applying the standard second-order accurate 5-point stencil (2D) and 7-point stencil (3D) central finite difference scheme  with a uniform mesh step size $h=1/N$, leads to a large-scale symmetric positive definite sparse linear system
\eq \label{Ax=b}
A_h u_h=b_h,
\ee
where $u_h$ denotes the finite difference approximation to the true solution $u$ over the set $\Omega_h$ of all interior grid points, $b_h$ encodes the source term $f$ and boundary data $g$, and $A_h$ has the following well-known 5-point (2D) or 7-point (3D) stencil representation
\begin{equation} \label{A_h}
	A_h = \frac{1}{h^2}\begin{bmatrix}
		& -1 &   \\
		-1 & 4& -1 \\
		& -1 &
	\end{bmatrix}\ \mbox{or}\ \frac{1}{h^2} \begin{Bmatrix}
	\begin{bmatrix}
		& 0 &   \\
		0 & -1& 0 \\
		& 0 &
	\end{bmatrix}
	\begin{bmatrix}
		& -1 &   \\
		-1 & 6& -1 \\
		& -1 &
	\end{bmatrix}
	\begin{bmatrix}
		& 0 &   \\
		0 & -1& 0 \\
		& 0 &
	\end{bmatrix}
\end{Bmatrix}
\end{equation}
For simplicity, we assume $f$ and $g$ to be sufficiently smooth such that the finite difference approximation $u_h$ by (\ref{Ax=b})  achieves a second-order accuracy in infinity norm, that is $\|u_h-u|_{\Omega_h}\|_\infty=O(h^2)$. For less regular $f$ and $g$, finite element discretization may be used to improve approximation accuracy in possible different weaker norms.

Due to the large condition number of the Laplacian matrix $A_h$ as the mesh step size $h$ is refined, the stationary iterative methods (e.g. Jacobi and Gauss-Seidel iterations) usually converge extremely slowly as the system size $n$ is increased. In contrast, multigrid methods can deliver mesh-independent convergence rate and the optimal computational complexity for solving the above linear system (\ref{Ax=b}),
where the choice of efficient and effective smoother is the key component.
In the past few decades, Poisson equation (\ref{Poisson}) has been numerically solved by various multigrid methods based on different smoothers and discretization schemes, see
\cite{braess1981contraction,braess1982convergence,braess1984convergence,holter1986vectorized,zhang1996acceleration,zhang1998fast,guillet2011simple,baker2011multigrid,notay2015massively,pan2017extrapolation,branden2022grid} and the references therein.
Though more effective than the weighted Jacobi smoother, the popular Gauss-Seidel smoother is not very friendly to  massively parallel computers due to its sequential nature \cite{adams2003parallel,notay2015massively}.
For general symmetric positive definite linear systems, significant efforts in the development of multigrid solvers have been concentrated on the design of effective parallelizable smoothers with smaller smoothing factors (and faster convergence rates), see for example \cite{birken2011designing,huang2021learning,brown2021tuning,CH2021addVanka,lottes2022optimal,Kraus2012,Yang2017} and the references therein.
In \cite{Brannick2015}, the authors compared three different Chebyshev polynomial smoothers in the context of aggressive coarsening,
where the one-dimensional  minimization formulations are defined over a finite interval that bounds all the eigenvalues of  diagonally preconditioned system.
In this paper, we will only focus on the development of effective and highly parallelizable SPAI smoothers,
whose convergence rates can be precisely estimated by local Fourier analysis (LFA), a quantitative tool to study multigrid convergence performance and optimize relaxation parameters.

Inspired by the  well-studied sparse approximate inverse preconditioners \cite{cosgrove1992approximate,benzi1996sparse,Grote1997,gould1998sparse,benzi1998sparse,tang1999toward,Benzi1999,saad2003iterative,bertaccini2018iterative} for preconditioning sparse linear systems, the class of so-called SPAI smoothers were widely studied \cite{Tang2000,Brker2001,Brker2002,Bolten2016} for general linear systems, where superior smoothing effects were achieved. Besides broad applicability, the inherent parallelism \cite{Grote1997,bertaccini2016sparse} in the framework of parallel computing is one major advantage of SPAI smoothers.
The construction of high quality SPAI preconditioners or smoothers of general linear systems is computationally expensive since it often requires to solve (multiple) norm minimization problems. However, for our considered well-structured linear systems, it is possible to analytically find the symbols of highly optimized and effective SPAI smoothers through LFA techniques, which completely avoids the expensive numerical construction procedure using various optimization formulations.

In this work, we use LFA to derive new SPAI smoothers for 2D and 3D Laplacian problems. Our proposed smoothers are more efficient and effective than these studied in \cite{Bolten2016,CH2021addVanka}.
In the literature,  LFA has been widely used to study different discretization and relaxation schemes for the Poisson equation \eqref{Poisson}. For example,
\cite{he2020two}  uses LFA to study Jacobi smoother of multigrid methods for higher‐order finite‐element approximations to the Laplacian problem.
In \cite{de2021two}, multiplicative
Schwartz smoothers are investigated by LFA for isogeometric discretizations of the Poisson equation. While, \cite{CH2021addVanka} studies additive Vanka-type smoothers.   Multigrid methods based on triangular grids with  standard-coarsening and three-coarsening strategies  for the Poisson equation  are studied in \cite{gaspar2009fourier,gaspar2009geometric}.
 {Our new SPAI smoothers are very efficient in inverting Laplacian.}

The whole paper is organized as follows.  In the next section we recall basic concepts and ideas of LFA that will be used in our analysis.
In Section 3,  new SPAI smoothers are developed and analyzed,
 where the technical proofs of our main theoretical results Theorem 3.1 and 3.2 are given.
In Section 4, several  numerical examples are reported to confirm
the LFA predicted multigrid convergence rates of our proposed SPAI smoothers.
Finally, some conclusions are given in Section 5.
\section{A brief review of LFA}
Local Fourier analysis (LFA) \cite{trottenberg2000multigrid,wienands2004practical} is the standard tool  to quantitatively predict the convergence rate of a given multigrid algorithm. In this section we briefly describe the main mechanism of how LFA works.
In literature of multigrid methods, LFA is very useful to study the performance of multigrid smoothers. For this,  a typical LFA procedure  includes the following three steps:
\begin{itemize}
	\item[(1)]  choose a good smoother (e.g. Jacobi) based on the system coefficient matrix;
	\item[(2)] analyze  LFA smoothing factor $\mu(\omega)$ of a $\omega$-parameterized relaxation scheme,
	and (exactly or approximately) find the best relaxation parameter $\omega_{\opt}$ that minimizes $\mu(\omega)$, which often requires very technical and tedious analysis;
	\item[(3)]  numerically verify the corresponding LFA two-grid convergence factor and the actual multigrid convergence rate and compare with the obtained optimal smoothing factor $\mu_{\opt}=\mu(\omega_{\opt})$.
\end{itemize}
To address some predictive limitation of the smoothing factor $\mu$, where coarse-grid correct plays an important role in multigrid performance,  the more sophisticated two-grid LFA convergence factor  provides a more reliable estimate of actual multigrid convergence rate. For the finite difference discretization considered here, the LFA smoother factor can offer a sharp prediction of actual multigrid performances. Thus, we will focus on optimizing the smoothing factor $\mu$ analytically and then checking the two-grid LFA convergence factor numerically (see below Table 1).
Clearly,  the choice of an efficient and effective smoother plays a decisive role in determining the practical convergence rate of the overall multigrid algorithm, and the selection of best relaxation parameter $\omega_{\opt}$ is highly dependent on the chosen smoother too. We aim to design and analyze fast sparse approximate inverse smoothers by LFA techniques.

In the standard (geometric) multigrid method for solving the linear system (\ref{Ax=b}), the most commonly used smoothers (e.g., damped Jacobi or Gauss-Seidel) have the following preconditioned Richardson iteration form
\eq \label{smoother}
u_h^{k+1}=u_h^{k}+\omega M_h(b_h-A_h u_h^k)=\underbrace{(I_h-\omega M_h A_h)}_{S_h(\omega)} u_h^{k}+\omega M_h b_h,
\ee
where $M_h$ approximates $A_h^{-1}$, $\omega\in\IR$ is a damping parameter to be determined,
and $S_h(\omega)$ is called relaxation error operator.
For example, the damped Jacobi smoother takes $M_h=\diag(A_h)^{-1}$ and the Gauss-Seidel smoother uses $M_h=\mathrm{tril}(A_h)^{-1}$, where $\mathrm{tril}(A)$ extracts the lower triangular part of $A$.
Let $\varrho(S_h(\omega))$ be the spectral radius of $S_h(\omega)$. Then the   above fixed-point iteration (\ref{smoother}) is asymptotically convergent if and only if $\varrho(S_h(\omega))<1$, which enforces some restrictions on the choice of $\omega$. To estimate the multigrid performance, we can examine the smoothing effects of $S_h(\omega)$, that is how effectively it reduces  the   high frequency components of approximation errors.
However, in practice, it is hard to directly compute or estimate $\varrho(S_h(\omega))$, instead, we can use LFA to study the smoothing properties of $S_h(\omega)$ via its Fourier symbol.

We first give some definitions of LFA following \cite{trottenberg2000multigrid}.
With the standard coarsening, the low and high frequencies are given by ${\boldsymbol{\theta}  }\in  T^{\rm{L}} =\left[-\frac{\pi}{2}, \frac{\pi}{2}\right)^d$, $\boldsymbol{\theta} \in  T^{\rm{H}} =\left[-\frac{\pi}{2}, \frac{3\pi}{2}\right)^d \setminus T^{\rm{L}}$, respectively.
We define the LFA \textit{smoothing factor} for   $S_h(\omega)$  as
\begin{equation}
	\mu_{\rm loc}(S_h(\omega)) := \max_{\boldsymbol{\theta} \in T^{\rm{H}}}\{\varrho(\widetilde{S}_h(\omega;\boldsymbol{\theta}))\},
\end{equation}
where {the matrix} $\widetilde{S}_h(\omega;\boldsymbol{\theta})$ is the symbol of $S_h(\omega)$ and $\varrho(\widetilde{S}_h(\omega;\boldsymbol{\theta}))$  {stands} for its spectral radius.
In particular, for the finite difference Laplacian operator considered here, the symbol  $\widetilde{S}_h(\omega;\boldsymbol{\theta})$ is just a scalar number.
We also define the optimal smoothing factor $\mu_{\rm opt}$ and the corresponding optimal relaxation parameter $\omega_{\rm opt}$ as
\begin{equation}
	\mu_{\rm opt} :=\min_{\omega\in\IR}\mu_{\rm loc}(S_h(\omega)),\qquad {\omega_{\rm opt}:=\argmin_{\omega\in\IR} \mu_{\rm loc}(S_h(\omega))}.
\end{equation}
In general, it is very difficult to analytically find the values of $\mu_{\rm opt}$  and $\omega_{\rm opt}$.  We point out that LFA assumes periodic boundary conditions and it does not consider the potential influence of different boundary conditions. However, in many applications, either the LFA smoothing factor or  two-grid LFA convergence factor offers sharp predictions of  problems with  other boundary conditions \cite{rodrigo2017validity}.

Let $\widetilde{A}_h(\bm\theta)$ and $\widetilde{M}_h(\bm\theta)$ be the scalar symbol of $A_h$
and $M_h$, respectively.  Then the symbol of  $S_h(\omega)=I_h-\omega M_h A_h$ is obviously given by (note $\widetilde I_h=1$)
\eq
 \widetilde{S_h}(\omega;\bm\theta)=\widetilde I_h-\omega \widetilde M_h \widetilde A_h=1-\omega \widetilde{A}_h(\bm\theta)\widetilde{M}_h(\bm\theta),
\ee
which leads to a min-max optimization for finding the optimal smoothing factor
\begin{equation}
	\mu_{\rm opt} =\min_{\omega\in\IR} \max_{\boldsymbol{\theta} \in T^{\rm{H}}}   | 1-\omega \widetilde{A}_h(\bm\theta)\widetilde{M}_h(\bm\theta) |.
\end{equation}
If there holds $0<\lambda_0\le \widetilde{A}_h(\bm\theta)\widetilde{M}_h(\bm\theta)\le \lambda_1$ for all ${\boldsymbol{\theta} \in T^{\rm{H}}}$, then it is easy to obtain
\begin{equation} \label{mu_opt}
	\mu_{\rm opt} =  | 1-\omega_{\opt} \lambda_0 |=| 1-\omega_{\opt} \lambda_1|= \frac{\lambda_1-\lambda_0}{\lambda_1+\lambda_0}=1-\frac{2}{1+\lambda_1/\lambda_0}<1
\end{equation}
with
\begin{equation} \label{omega_opt}
\omega_{\opt}=2/(\lambda_0+\lambda_1).
\end{equation}
Hence $\mu_{\rm opt}$  is an increasing function of  $\lambda_1/\lambda_0\in [1,\infty)$ and a smaller ratio (or spectral condition number) $\lambda_1/\lambda_0$ gives a smaller smoothing factor. The best choice of $\widetilde{M}_h(\bm\theta)$ highly depends on the expression of $\widetilde{A}_h(\bm\theta)$ and hence the structure of matrix $A_h$.
Mathematically, one may suggest to select $\widetilde{M}_h(\bm\theta)=1/\widetilde{A}_h(\bm\theta)$ such that $\lambda_1/\lambda_0=1$ leading to	$\mu_{\rm opt}=0$, which is however impractical in computation since it requires the dense matrix $M_h=A_h^{-1}$.
Nevertheless, it is still desirable to choose $\widetilde{M}_h(\bm\theta)\approx 1/\widetilde{A}_h(\bm\theta)$ for all ${\boldsymbol{\theta} \in T^{\rm{H}}}$ such that $\lambda_1/\lambda_0$ is minimized  in certain  sense and meanwhile the matrix-vector product $M_h r_h$ with the residual vector $r_h:=b_h-A_h u_h^k$  is efficient to compute.  In the next section, we will construct effective $M_h$ by minimizing $\mu_{\opt}$ over a class of predefined symmetric stencil pattern that leads to parameterized symbol $\widetilde{M}_h(\bm\theta)$. {From \eqref{mu_opt},   we only need to either minimize $\lambda_1/\lambda_0$ or maximize $\lambda_0/\lambda_1$ to obtain the optimal smoothing factor, and the corresponding optimal $\omega$ is given by \eqref{omega_opt}.}
We first notice that a scalar multiple of $\widetilde{M}_h(\bm\theta)$ does not change $\mu_{\opt}$,
but it indeed leads to a rescaled $\omega_{\opt}$, hence one can normalize the symbol to simplify the analysis.
\section{New optimized SPAI smoothers}

\subsection{2D case}
For the 2D 5-point stencil $A_h$ given in (\ref{A_h}), its symbol reads
\begin{equation}\label{eq:symbol-Lh2D}
	\widetilde{A}_h(\bm\theta)=\frac{1}{h^2}(4-e^{i\theta_1}-e^{-i\theta_1}-e^{i\theta_2}-e^{-i\theta_2})=\frac{2}{h^2}(2-\cos \theta_1 -\cos \theta_2).
\end{equation}
The weighted point-wise Jacobi smoother $M_{J}=\diag(A_h)^{-1}$ has a  singleton stencil
\begin{equation}
	M_{J}= \frac{h^2}{4}\begin{bmatrix}
		& 0 &    \\
		0 & 1& 0 \\
		& 0 &
	\end{bmatrix}
\end{equation}
with  symbol $\widetilde{M}_{J}(\bm\theta)=h^2/4$,
which was shown to achieve the optimal smoothing factor  $\mu_{\opt}=3/5=0.6$ with $\omega_{\opt}=4/5=0.8$ \cite{trottenberg2000multigrid}.
Although $M_J$ is very easy to parallelize, its convergence rate of $0.6$ becomes rather slow and inefficient for large-scale systems.

In the seminar work \cite{Tang2000}, the authors derived a  5-point  SPAI smoother
\begin{equation}
	M_{\textrm{5,TW}}= \frac{h^2}{61}\begin{bmatrix}
		& 3 &    \\
		3 & 17& 3 \\
		& 3 &
	\end{bmatrix}
\end{equation}
with its symbol $\widetilde{M}_{\textrm{5,TW}}(\bm\theta)=h^2(17+6\cos\theta_1+6\cos\theta_2)/61$,
which was shown to have a smoothing factor of $21/61\approx 0.344$ if choosing $\omega=1$.
In fact, LFA shows this particular smoother $M_{\textrm{5,TW}}$ can achieve the optimal smoothing factor $\mu_{\opt}\approx 0.273$ if taking $\omega_{\opt}\approx 1.108$. Nevertheless, such a 5-point SPAI smoother can be further improved to obtain a smaller $\mu_{\opt}$. Specifically, in \cite{Brandt1977,Brker2001,Bolten2016}, the authors obtained the following best 5-point SPAI smoother (among all symmetric 5-point stencils)
\begin{equation}
	M_{\textrm{5}}= \frac{8h^2}{41}\begin{bmatrix}
		& 1 &    \\
		1 & 6& 1 \\
		& 1 &
	\end{bmatrix}
\end{equation}
with its symbol $\widetilde{M}_{\textrm{5}}(\bm\theta)=(48+16\cos\theta_1+16\cos\theta_2)/41$,
which gives the  optimal smoothing factor  $\mu_{\opt}=9/41\approx 0.220$ with $\omega_{\opt}=1/4= 0.25$.
%Notice that  re-scaled $M_{\textrm{5}}$ will give the same \textit{minimal}  $\mu_{\opt}$ with re-scaled  $\omega_{\opt}$.
One may wondering  can we construct a new SPAI smoother with a   smaller optimal smoothing factor $\mu_{\opt}$?

%{\color{red} What's the word  `minimal' do you mean?}

The short answer is \textit{yes}, but we have to search from those SPAI smoothers with wider/denser stencil. In this paper, for the sake of computational efficiency, we consider a general class of symmetric 9-point stencil smoothers of the following form
\[
\Upsilon(\alpha,\beta,\gamma):={h^2}\begin{bmatrix}
	\gamma	& \beta & \gamma \\
	\beta & \alpha& \beta \\
	\gamma	& \beta & \gamma
\end{bmatrix},
\]
where $\alpha,\beta,$ and $\gamma$ are to be optimized through LFA.
Obviously, the special case $\Upsilon(\alpha,\beta,0)$ with $\gamma=0$ reduces to the above 5-point stencil and hence we expect a smaller optimal smoothing factor with the free choice of $\gamma$. The symbol of $\Upsilon$ reads
\[\widetilde{\Upsilon}(\bm\theta)=h^2(\alpha+2\beta\cos\theta_1+2\beta\cos\theta_2+4\gamma\cos\theta_1\cos\theta_2).\]
Based on the idea of domain decomposition method, an  element-wise additive Vanka smoother corresponding to the above 9-point stencil was proposed recently in  \cite{CH2021addVanka}
\begin{equation}
	M_{\mathrm{Vanka}} = \frac{h^2}{96}\begin{bmatrix}
	1	& 4 & 1  \\
	4 & 28& 4 \\
	1	& 4 & 1
	\end{bmatrix}=\frac{1}{96}\Upsilon(28,4,1),
\end{equation}
 with $\widetilde{M}_{\mathrm{Vanka}}=(28+8\cos\theta_1+8\cos\theta_2+4\cos\theta_1\cos\theta_2)/96$, which gives the optimal smoothing factor $\mu_{\opt}=7/25=0.28$ with $\omega_{\opt}=24/25$.
Though with better performance than weighted Jacobi smoother, such a Vanka 9-point stencil smoother $M_{\mathrm{Vanka}}$ can be greatly improved to attain a smaller optimal smoothing factor than the above best 5-point SPAI smoother $M_5$, which is one of our major contribution.

To find the best 9-point stencil  $\Upsilon(\alpha,\beta,\gamma)$ that achieves the optimal smoother factor $\mu_{\rm opt}$, we essentially need to solve the following min-max optimization problem
\bean
&&\mu_{\rm opt}=\min_{\alpha,\beta,\gamma,\omega}\mu_{\rm loc} =\min_{\alpha,\beta,\gamma,\omega} \max_{\boldsymbol{\theta} \in T^{\rm{H}}}   | 1-\omega \widetilde{A}_h(\bm\theta) \widetilde{\Upsilon}(\bm\theta)|\\
&&=\min_{\alpha,\beta,\gamma,\omega} \max_{\boldsymbol{\theta} \in T^{\rm{H}}}   | 1-\omega(2\alpha+4\beta(\cos\theta_1+\cos\theta_2)+8\gamma\cos\theta_1\cos\theta_2) (2-\cos \theta_1 -\cos \theta_2) |\nonumber\\
&&=\min_{a,b ,\omega} \max_{\boldsymbol{\theta} \in T^{\rm{H}}}   | 1-\omega(b+a(\cos\theta_1+\cos\theta_2)+\cos\theta_1\cos\theta_2) (2-\cos \theta_1 -\cos \theta_2)
%&&=:\min_{a,b } J(a,b),
\eean 
where we fixed $8\gamma=1$ and set $b=2\alpha$ and $a=4\beta $ to simplify the discussion. Notice that the normalization condition $8\gamma=1$ only changes the choice of $\omega_{\opt}$ and it will not affect $\mu_{\opt}$ since the ratio $\lambda_1/\lambda_0$ remains the same.
Based on some technical analysis as stated in the following theorem,
we can obtain the following best 9-point SPAI smoother  $M_9$
that achieves the optimal smoothing factor
\begin{equation}
M_9 =\frac{1}{24}\Upsilon(44,10,3)= \frac{h^2}{24}\begin{bmatrix}
		3	& 10 & 3  \\
		10 & 44& 10 \\
		3	& 10 & 3
	\end{bmatrix}.
\end{equation}

\begin{theorem} \label{Thm_SAI92D}
Among all symmetric $9$-point stencil of form $\Upsilon(\alpha,\beta,\gamma)$, the SPAI smoother $M_9=\frac{1}{24}\Upsilon(44,10,3)$ gives the optimal smoothing factor \[\mu_{\opt}={(9+8\sqrt{10})}/{215}\approx 0.1595\]
with  the optimal relaxation parameter \[\omega_{\opt}={309-12\sqrt{10}\over1720}\approx 0.1576.\]
\end{theorem}
\begin{proof}
For better exposition, the rather technical proof is given in Appendix A.	
\end{proof}
Compared to  $M_5$ (with $\mu_{\opt}\approx 0.2195$),
the proposed
$M_9$ (with $\mu_{\opt}\approx 0.1595$) reduces $\mu_{\opt}$ by about 30\%, which hence leads to a faster multigrid convergence rate. Clearly, $M_9$ also provides much faster convergence rate than $M_{\mathrm{Vanka}}$ with the same cost.
 {Recall the 2D pointwise lexicographic Gauss-Seidel smoother only gives $\mu_{\opt}=0.5$ \cite{trottenberg2000multigrid,Hocking2011}.}

It is worthwhile to point out that the LFA technique can be applied to other discretization to identify  optimal smoother.
For example, the authors in \cite{Bolten2016} studied a 9-point stencil arising from linear finite element method for the 2D Poisson equation
\[
\widehat A_h = \frac{1}{h^2}\begin{bmatrix}
	-1& -1 &  -1 \\
	-1 & 8& -1 \\
	-1& -1 &  -1
\end{bmatrix},
\]
where the best SPAI smoother with a 9-point stencil (of the reduced form $\Upsilon(\alpha,\beta,\beta)$)
\[
\widehat M_9=\frac{4h^2}{153}\begin{bmatrix}
	1& 1 &  1\\
	1 & 10& 1 \\
	1& 1 & 1
\end{bmatrix}
\]
gives the optimal smoothing factor $\mu_{\opt}=1/17\approx 0.0588$ with $\omega_{\opt}=1$. 
%Though appealing, this 2D 9-point stencil Laplacian matrix is more expensive than the 2D 5-point stencil Laplacian matrix $A_h$ while also attaining the same second-order accuracy  
We emphasis that it is technically more difficult to {theoretically} find $M_9$ than both $M_5$ and $\widehat M_9$, since the product symbol $\widetilde{M_9}(\bm\theta)\widetilde{A_h}(\bm\theta)$ is more complicated than both $\widetilde{M_5}(\bm\theta)\widetilde{A_h}(\bm\theta)$ and $\widetilde{\widehat M_9}(\bm\theta)\widetilde{\widehat A_h}(\bm\theta)$ due to mis-matched stencil patterns (5-point verse 9-point). {However, for given structured smoother, one can apply some  robust optimization approaches, for example, \cite{brown2021tuning}, to numerically identify the (approximate) optimal smoother.}
\subsection{3D case}

For the 3D 7-point stencil $A_h$ given in (\ref{A_h}), its symbol reads
\begin{equation}\label{eq:symbol-Lh3D}
	\widetilde{A}_h(\bm\theta)=\frac{2}{h^2}(3-(\cos \theta_1 +\cos \theta_2+\cos \theta_3)).
\end{equation}
The simplest Jacobi smoother is $M_J=(h^2/6)I_h$ with $\widetilde{M}_h(\bm\theta)=h^2/6$ and it achieves a quite large  $\mu_{\opt}=5/7\approx 0.714$ with the optimal damping parameter $\omega_{\opt}=6/7$ \cite{trottenberg2000multigrid}.

Inspired by $M_5$, we look at all 7-point stencil SPAI smoothers of the form
\begin{equation}
	\mathcal{T}(\alpha,\beta):=  \frac{h^2}{4} \begin{Bmatrix}
		\begin{bmatrix}
			& 0 &   \\
			0 & \beta& 0 \\
			& 0 &
		\end{bmatrix}
		\begin{bmatrix}
			& \beta &   \\
			\beta & 2\alpha& \beta \\
			& \beta &
		\end{bmatrix}
		\begin{bmatrix}
			& 0 &   \\
			0 & \beta& 0 \\
			& 0 &
		\end{bmatrix}
	\end{Bmatrix}
\end{equation}
with a symbol $\widetilde{\mathcal{T}}(\bm\theta)= (h^2/4)(2\alpha+2\beta(\cos\theta_1+\cos\theta_2+\cos\theta_3))$.
Hence there holds
\[
\widetilde{\mathcal{T}}(\bm\theta)\widetilde{A}_h(\bm\theta)=
\left(\alpha+\beta(\cos\theta_1+\cos\theta_2+\cos\theta_3)\right) \left(3-(\cos \theta_1 +\cos \theta_2+\cos \theta_3)\right).
\]
Following the techniques in \cite{Bolten2016}, we can find the following best 7-point SPAI smoother
\begin{equation}
	M_{\mathrm{7}}= \frac{h^2}{10}\begin{Bmatrix}
		\begin{bmatrix}
			& 0 &   \\
			0 & 1& 0 \\
			& 0 &
		\end{bmatrix}
		\begin{bmatrix}
			& 1 &   \\
			1 & 8& 1 \\
			& 1 &
		\end{bmatrix}
		\begin{bmatrix}
			& 0 &   \\
			0 & 1& 0 \\
			& 0 &
		\end{bmatrix}
	\end{Bmatrix}=\mathcal{T}(8/5,2/5),
\end{equation}
which gives the   optimal smoothing factor as stated in the following theorem.
\begin{theorem}
	Among all symmetric $7$-point stencil of form $\mathcal{T}(\alpha,\beta)$, the SPAI smoother $M_7=\mathcal{T}(8/5,2/5)$ gives the  optimal smoothing factor \[\mu_{\opt}=25/73\approx 0.343\]
	with the optimal relaxation parameter \[\omega_{\opt}=20/73\approx 0.274.\]
\end{theorem}
\begin{proof}
 By normalization, we assume without loss of generality that $\beta=2/5$ and $\alpha=a\beta$.
Let $x_i=\cos\th_i$ with $i=1,2,3$, and transform the high frequency variable $\bm\theta\in T^{\rm{H}}$ to $(x_1,x_2,x_3)\in X^{\rm{H}}:=[-1,1]^3\setminus(0,1]^3$. {By standard calculation, we have
\[\mu_{\rm loc}=\max_{(x_1,x_2,x_3)\in X^H}|1-2/5\cdot\om f_a(x_1,x_2,x_3)|,\] }
where
\[f_a(x_1,x_2,x_3):=(a+x_1+x_2+x_3)(3-x_1-x_2-x_3).\]
Denote the extreme values of $f_a$ as
\[\chi_a:=\max_{(x_1,x_2,x_3)\in X^{\rm{H}}}f_a(x_1,x_2,x_3),~~m_a:=\min_{(x_1,x_2,x_3)\in X^{\rm{H}}}f_a(x_1,x_2,x_3).\]
{From \eqref{mu_opt}, to obtain the optimal smoothing factor, we  only need to maximize the ratio $r_a:=m_a/\chi_a$.}

We restrict $a$ in the region $R:=\{a\in\R:~m_a>0\}$ to guarantee the convergence of the relaxation scheme.
Note that $f_a(x_1,x_2,x_3)=V_a(t):=(a+t)(3-t)$ with $t:=x_1+x_2+x_3\in[-3,2]$.
Using $f_a>0$, we have $a>3$. Since $V_a(t)$ is a concave function in $t$ (i.e., $V_a''(t)<0$) with a unique critical point at $t=(3-a)/2$, we obtain
\[m_a=\min\{V_a(-3),V_a(2)\}=\min\{6a-18,a+2\},\]
and
\[
	\chi_a=\begin{cases}
		(a+3)^2/4,&~3<a\le 9,\\
		6a-18,&~a\ge 9.
	\end{cases}
\]

If $3<a\le4$, then
$r_a={6a-18\over (a+3)^2/4}\le{24\over49}.$
If $4\le a\le 9$, then
$r_a={a+2\over (a+3)^2/4}\le{24\over49}.$
If $a\ge9$, then
$r_a={a+2\over6a-18}\le{11\over36}.$
A combination of the above three cases gives $r_a\le r_4=24/49$ with $m_4=6$ and $\chi_4=49/4$.
{Hence, from \eqref{mu_opt} we obtain
\[\mu_{\opt}={\chi_4-m_4\over \chi_4+m_4}={25\over73}\approx 0.343\] }
with optimal $\om_{\opt}$ satisfied $\beta\om_{\opt}={2/(\chi_4+m_4)}$, where $\beta=\frac{2}{5}$, that is
\[ \om_{\opt}={2\over \beta(\chi_4+m_4)}=\frac{5}{2}{8\over73}={20\over73}\approx 0.274.\]

\end{proof}
Similar to the above discussed 2D cases (from $M_5$ to $M_9$), one can obtain a smaller optimal smoothing factor if
 a wider/denser (e.g. 19-point or 27-point) symmetric stencil is used to construct the SPAI smoother for 3D problem.
Nevertheless, the resulting product symbol becomes much more complicated , which is overwhelmingly tedious to optimize analytically.
In such situations, one may resort to some robust optimization approaches (based on the same min-max formulation), see e.g. \cite{brown2021tuning}.
Although further discussion is beyond the scope of this paper, we numerically verified such heuristically optimized 19-point or 27-point SPAI smoothers indeed deliver slightly faster convergence rates than that of $M_7$.
 {We also mention the 3D pointwise lexicographic Gauss-Seidel smoother only attains $\mu_{\opt}=(4+\sqrt{5})/11\approx 0.567$ \cite{Hocking2011}, which is much larger than that of the above optimized SPAI smoother $M_7$.}

\section{Numerical results}
In this section, we present some numerical tests to illustrate the effectiveness of our proposed multigrid algorithms. All simulations are  implemented with MATLAB on {a Dell Precision 5820 Workstation with Intel(R) Core(TM) i9-10900X CPU@3.70GHz and 64GB RAM},
where the CPU times (in seconds) are estimated by the timing functions \texttt{tic/toc}.
In our multigrid algorithms, we use the coarse operator from re-discretization with a coarse mesh step size $H=2 h$, full weighting restriction and linear interpolation operators, W or V cycle with $\nu_1$-pre and $\nu_2$-post smoothing iteration, the coarsest mesh step size $h_0=1/4$, and the stopping tolerance $tol=10^{-10}$ based on reduction in relative residual norms.
We will test both $W(\nu_1+\nu_2)=W(1+0)$ and $V(\nu_1+\nu_2)=V(1+1)$ cycles in our numerical examples.
The initial guess $u_h^{(0)}$ is chosen as uniformly distributed random numbers in $(0,1)$.
The multigrid convergence rate (factor) of $k$-th iteration  is computed as \cite{trottenberg2000multigrid}
\eq
\widehat\rho^{(k)}=\left( \|r_k\|_2/\|r_0\|_2 \right)^{1/k},
\ee
where $r_k=b_h-A_h u_h^{(k)}$ denotes the residual vector after the $k$-th multigrid iteration. We will report  $\widehat\rho^{(k)}$ of the last multigrid iteration  as the actual convergence rate.
The MATLAB codes for reproducing the following figures are available online at the link:
\url{https://github.com/junliu2050/SPAI-MG-Laplacian}.

Let  $\nu=\nu_1+\nu_2$ be the total number of smoothing steps in each multigrid cycle.
In practice, the LFA smoothing factor often offers a sharp prediction of LFA two-grid convergence factor $\rho_h$ and actual two-grid performance, which also predicts the W-cycle multigrid convergence rate \cite{wienands2004practical,trottenberg2000multigrid}.
In  Table \ref{tab:mu-rho-results-256}, we numerically optimize the LFA two-grid convergence factor $\rho_h(\nu=1)$  with respect to the relaxation parameter $\omega\in(0,1]$, and then use the numerically obtained optimal parameter $\omega^{TG}_{\rm opt}$ to compute  the corresponding smoothing factor $\mu(\omega^{TG}_{\rm opt})$, and $\rho_h(\nu)$ as a function of increasing $\nu=2,3,4$. We observe that two-grid LFA convergence factor $\rho_h(\nu=1)$ is the same as the LFA smoothing factor $\mu(\omega^{TG}_{\rm opt})$, and the approximately optimal $\omega^{TG}_{\rm opt}$ and $\rho_h(\nu=1)$ match with  our theoretical  smoothing analysis, $\omega_{\rm opt}$ and $\mu_{\rm opt}$,  respectively.  As compared in Table \ref{tab:mu-rho-results-256}, we also include the damped Jacobi smoother $M_J$ and the SPAI smoother  $M_5$.
Both our proposed SPAI smoothers $M_9$ and $M_7$ significantly outperform the Jacobi smoother $M_J$, which are also confirmed by the following several 2D and 3D numerical examples.

	\begin{table}[H]
	\caption{LFA predicted two-grid convergence factor $\rho_h(\nu)$  using  $\omega^{TG}_{\opt}$   obtained from numerically minimizing two-grid LFA convergence factor $\rho_h(\nu=1)$ and the corresponding LFA  smoothing factor $\mu(\omega^{TG}_{\rm opt})$ with $h=\frac{1}{256}$ (for 2D) and $h=\frac{1}{64}$ (for 3D).}
	\centering
	\begin{tabular}{|c|l||c|c||cccc|}
		\hline
		&&$\omega^{TG}_{\rm opt}$  & $\mu(\omega^{TG}_{\rm opt})$  &{$\rho_h(\nu=1)$}  & {$\rho_h(\nu=2)$ }  &{$\rho_h(\nu=3)$}   & {$\rho_h(\nu=4)$}  \\ \hline
			\multirow{3}{*}{$2D$}
		 &{$M_J$}   &0.800        &0.600      &0.600       &0.360      & 0.216   & 0.137   \\
		 &{$M_5$}   &0.250        &0.220     &  0.220   &   0.087     & 0.056  &  0.044 \\
		&	 {$M_9$}   &  0.158       &0.160     &0.160     & 0.070     &  0.046  & 0.035\\ \hline
			\multirow{2}{*}{$3D$}
			&{$M_J$}   &    0.857    &  0.714     & 0.714       &  0.510  &   0.364  &  0.260  \\
			&{$M_7$}   &   0.274     &  0.343   &  0.343   &   0.152     &  0.107 &  0.085  \\
		%	&	 {$M_{27}$}   &   0.525       & 0.197    &    0.197  &   0.105  &  0.073  & 0.057 \\
		\hline
	\end{tabular}\label{tab:mu-rho-results-256}
\end{table}

\subsection{Example 1 \cite{Briggs2000}}
In the first example we consider the following data
\[
u=(x^2-x^4)(y^4-y^2),\  f=2(1-6x^2)(y^2-y^4)+2(1-6y^2)(x^2-x^4),\ g=0.
\]
In Fig. \ref{fig_MGplot_2D_Ex1_WV}, we compare the multigrid convergence performance of our considered three  SPAI-type smoothers: $M_J$, $M_5$,
 and $M_9$, where the estimated convergence rates match with the LFA preconditions shown in Table \ref{tab:mu-rho-results-256}.  Moreover,  Fig. \ref{fig_MGplot_2D_Ex1_WV} reveals  that using V-cycle is much cheaper than W-cycle. Clearly, both SPAI smoothers $M_5$ and $M_9$ attain significantly faster convergence rates and also cost less CPU times than the Jacobi smoother $M_J$.
In serial computation, we only observe marginal speed up in CPU times for $M_9$ over $M_5$, since $M_9$ has a wider stencil and higher operation cost in each iteration.
But we expect to achieve even more significant speedup in parallel computation since SPAI smoothers are embarrassingly parallelizable and the parallel CPU times will be mainly determined by the required sequential iteration numbers.

\begin{figure}[htp!]
	\centering
	\includegraphics[width=0.49\textwidth]{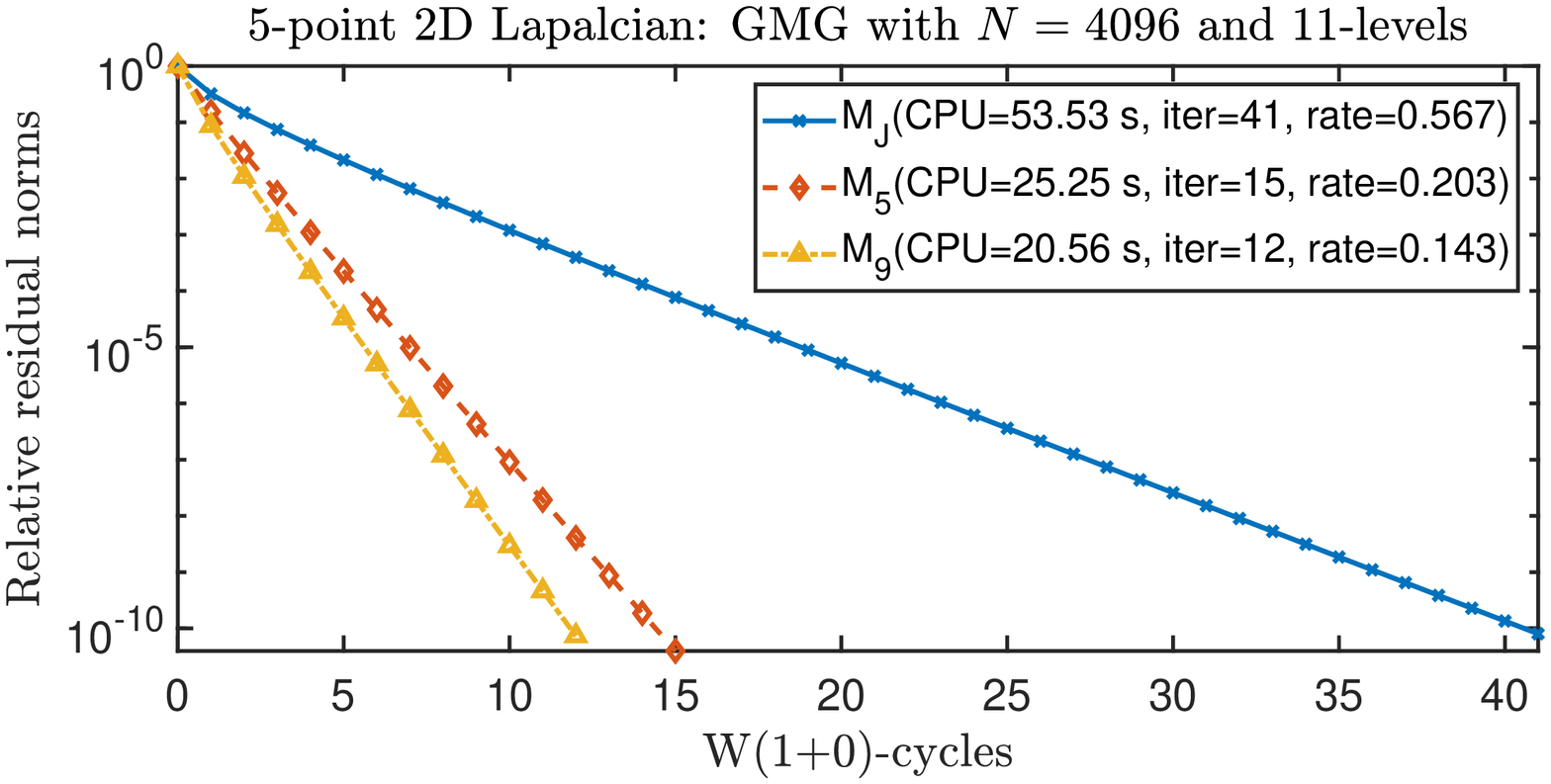}
	\includegraphics[width=0.49\textwidth]{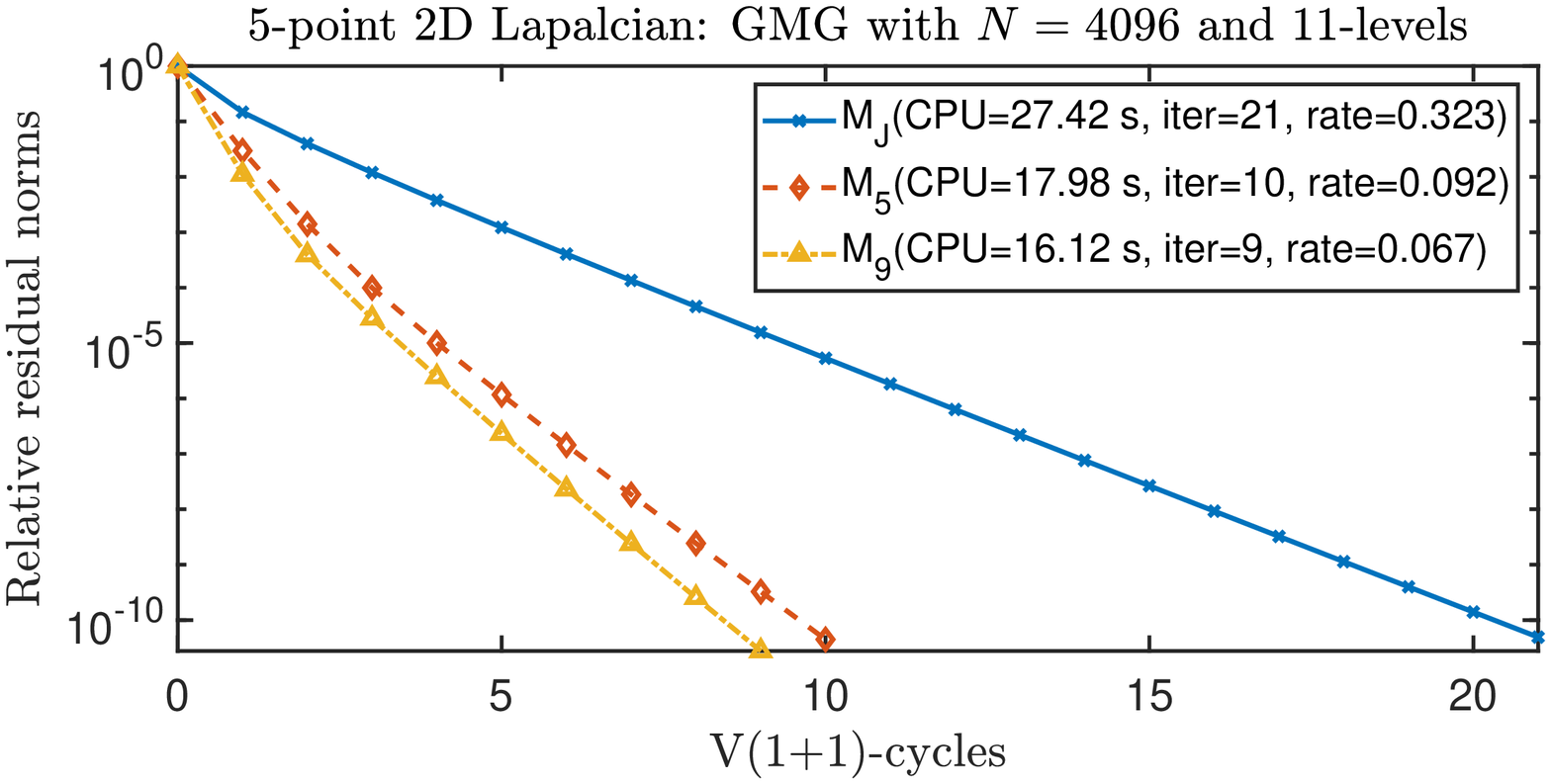}
	\caption{Example 1: comparison of  multigrid convergence with different smoothers. Left: W(1+0)-cycle. Right: V(1+1)-cycle.} \label{fig_MGplot_2D_Ex1_WV}
\end{figure}

\subsection{Example 2}
In the second example we consider the following data
\[
u=x\ln(x)y\ln(y),\  f=-x\ln(x)/y-y\ln(y)/x,\ g=0,
\]
where $f$ has singularity near the boundary with $x=0$ or $y=0$.
In Fig. \ref{fig_MGplot_2D_Ex2_WV}, we compare the multigrid convergence performance of our considered three SPAI-type smoothers: $M_J$, $M_5$,
and $M_9$, where the observed convergence rates are the same as those reported in Example 1. Again, we see V-cycle multigrid is more efficient than W-cycle multigrid.
This example shows that the convergence rates of SPAI smoothers are not obviously influenced by the lower regularity of the given source term $f$. 

\begin{figure}[htp!]
	\centering
	\includegraphics[width=0.49\textwidth]{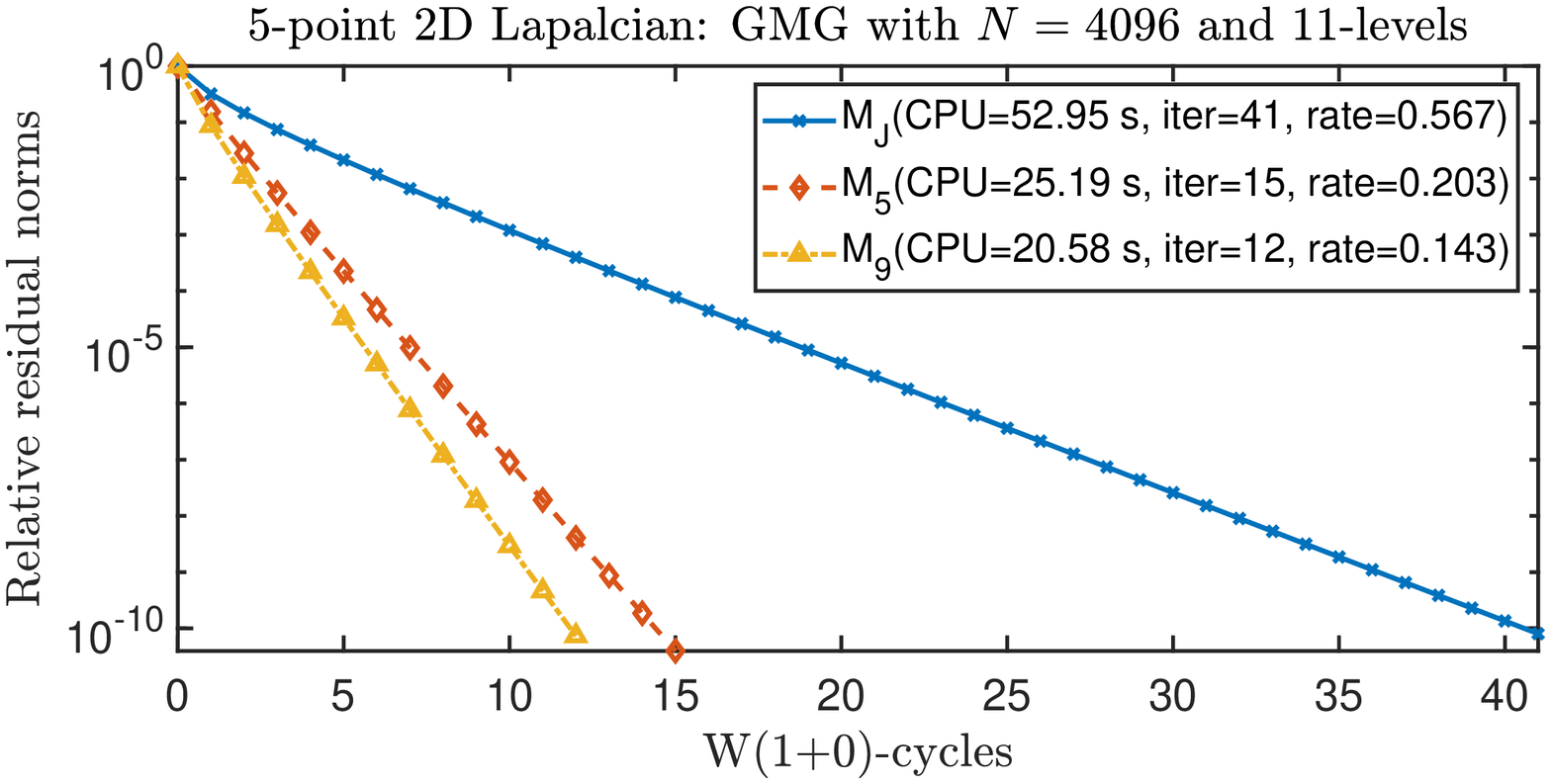}
	\includegraphics[width=0.49\textwidth]{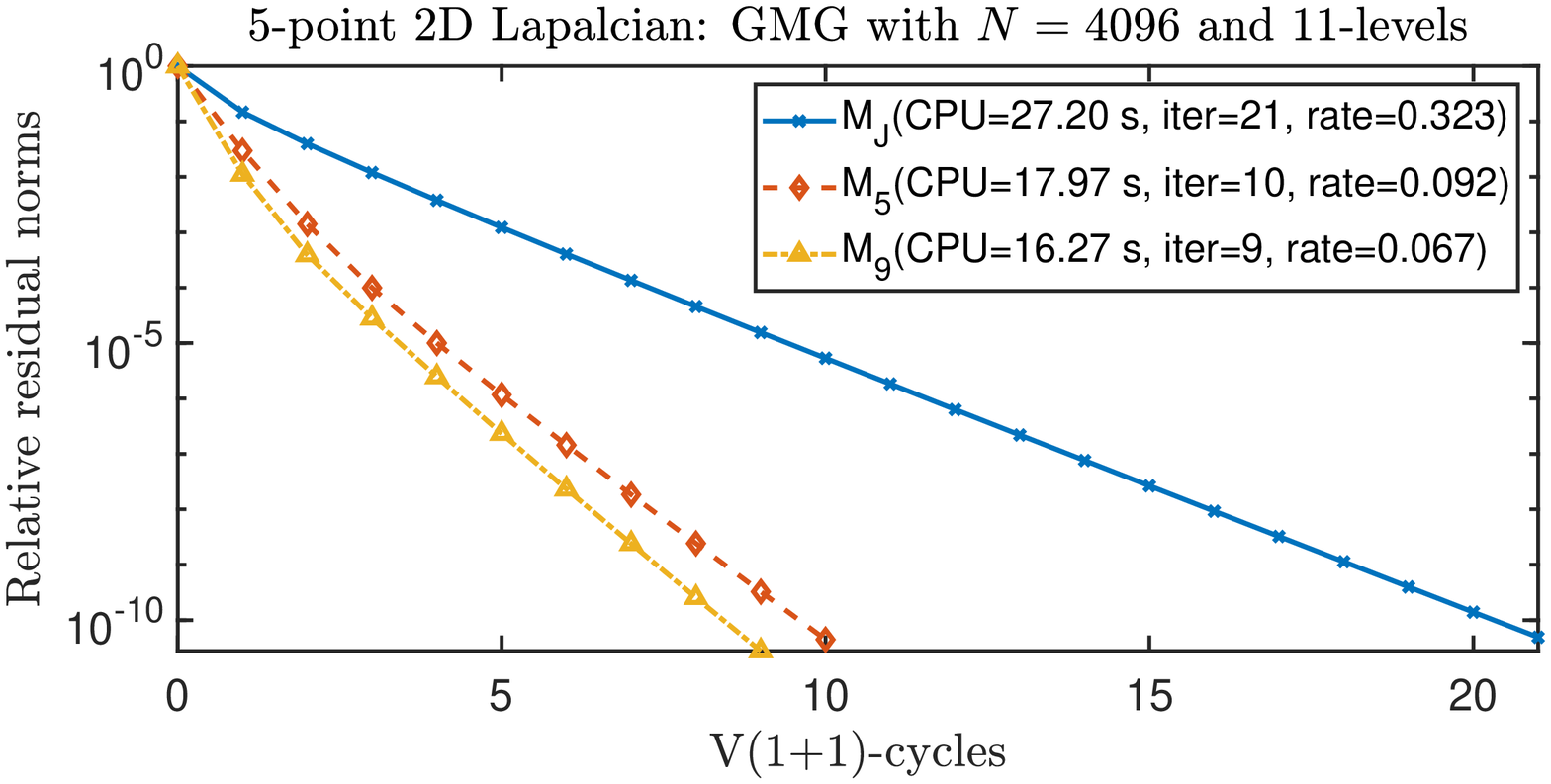}
	\caption{Example 2: comparison of  multigrid convergence with different smoothers. Left: W(1+0)-cycle. Right: V(1+1)-cycle.} \label{fig_MGplot_2D_Ex2_WV}
\end{figure}

\subsection{Example 3 \cite{zhang1998fast}}
In the third example, we consider the following 3D data
\[
u=\sin(\pi x)\sin(\pi y)\sin(\pi z),\ f=3\pi^2\sin(\pi x)\sin(\pi y)\sin(\pi z),\ g=0.
\]
In Fig. \ref{fig_MGplot_3D_WV}, we compare the multigrid convergence performance of  the two SPAI-type smoothers: $M_J$ and $M_7$, where the observed convergence rates are compatible with the LFA predictions presented in   Table \ref{tab:mu-rho-results-256}.
For both W and V cycles,  $M_7$ takes about half of the CPU times by $M_J$, as predicted by Table \ref{tab:mu-rho-results-256}. For all smoothers, we see that V-cycle is more efficient that W-cycle. We highlight that with $N=512$ the system has about 134 million unknowns,
which takes about 200 seconds for $M_7$ with V-cycle.
\begin{figure}[htp!]
	\centering
	\includegraphics[width=0.49\textwidth]{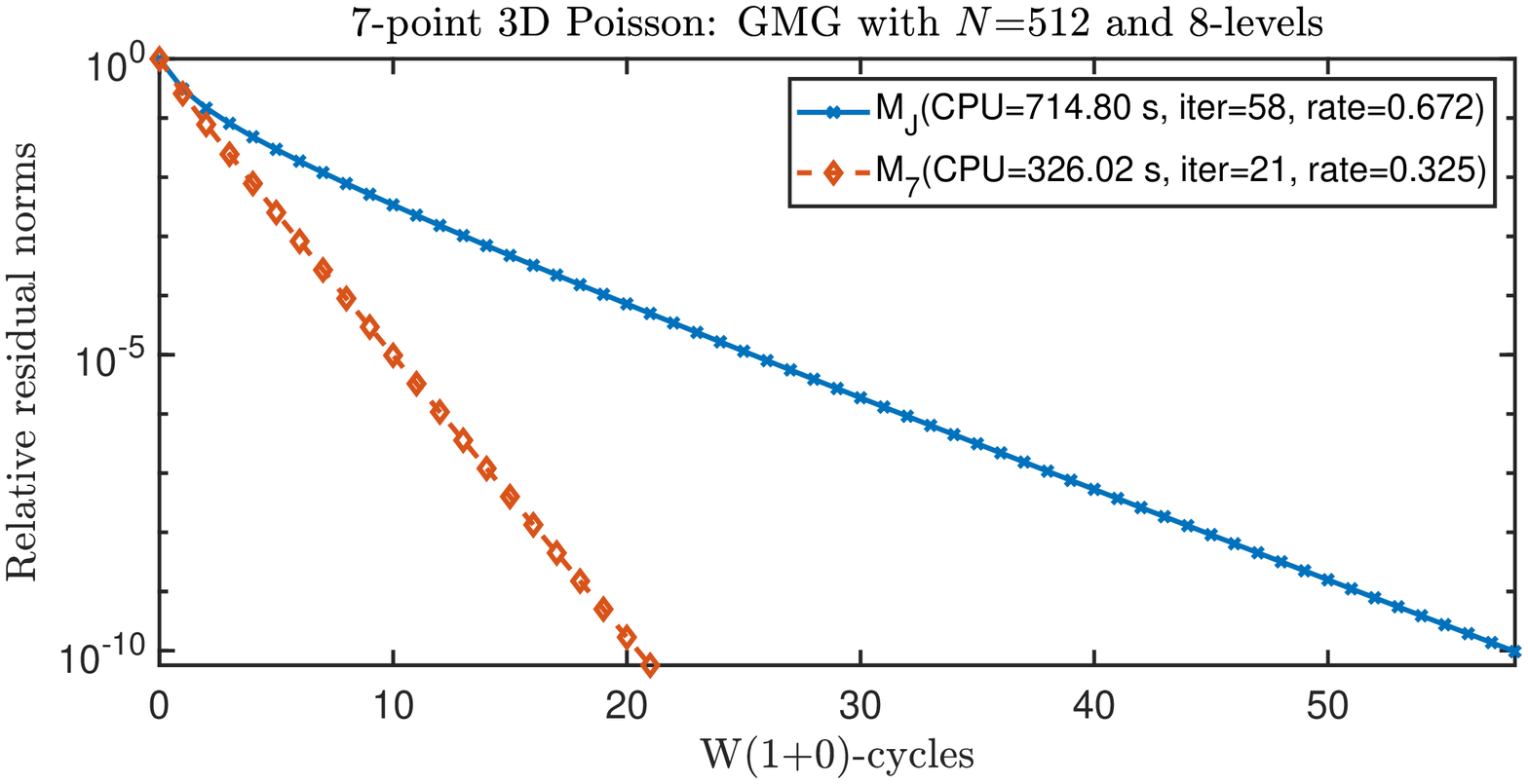}
	\includegraphics[width=0.49\textwidth]{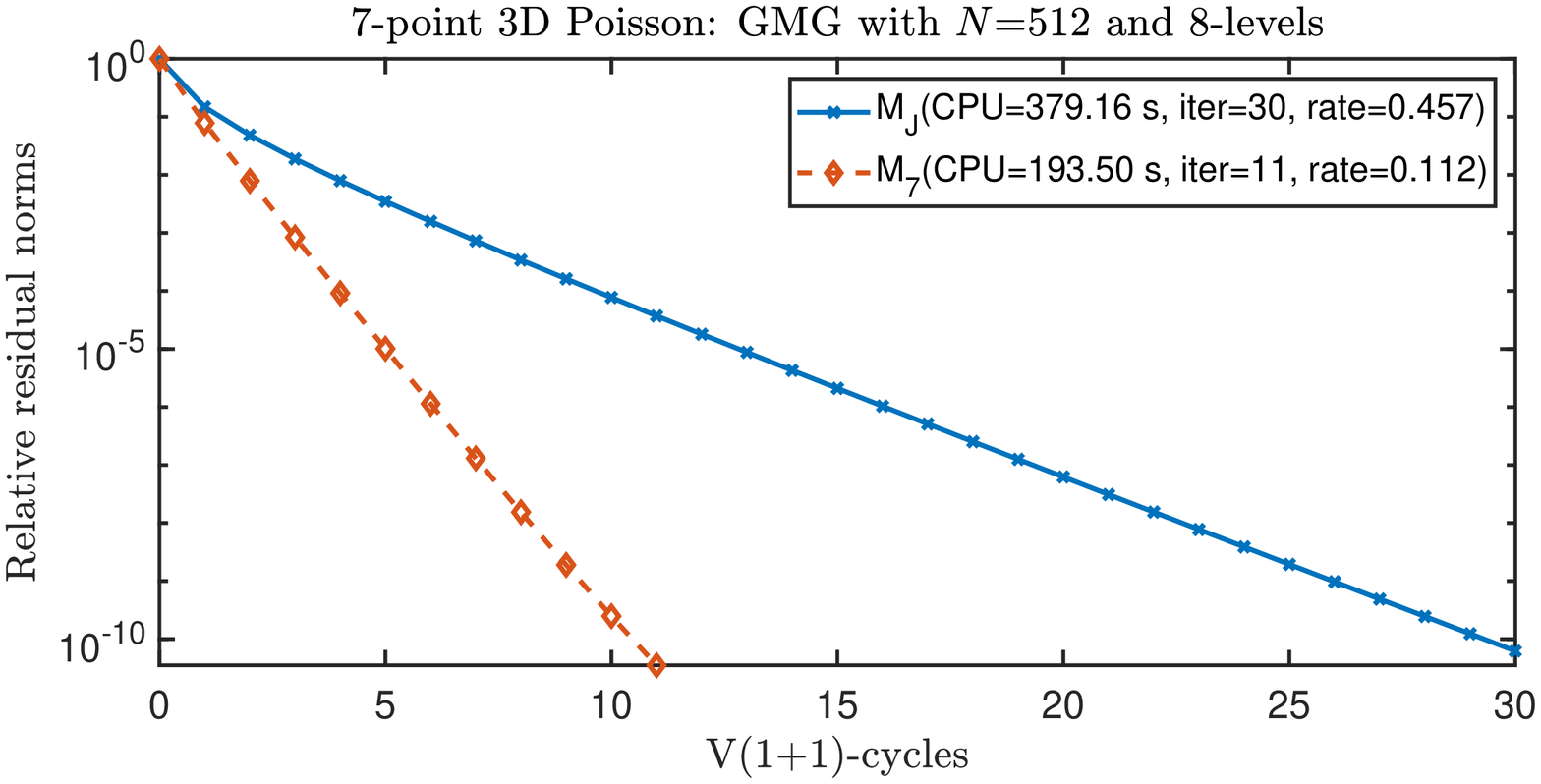}
	\caption{Example 3: comparison of  multigrid convergence with different smoothers. Left: W(1+0)-cycle. Right: V(1+1)-cycle.} \label{fig_MGplot_3D_WV}
\end{figure}

\section{Conclusion}
In this paper, we proposed and analyzed new 9-point and 7-point stencil based SPAI multigrid smoothers for solving 2D and 3D Laplacian linear systems respectively. The obtained optimal LFA smoothing factors are significantly smaller than that of the state-of-the-art SPAI smoothers in literature.  {The crucial optimal relaxation parameters are exactly derived through rigorous analysis.}  Numerical results with 2D and 3D examples validated our theoretical analysis and demonstrated the effectiveness of our proposed SPAI smoothers.
It is interesting to extend such SPAI smoothers to  general   second-order elliptic PDEs with variable coefficients, see e.g. \cite{nielsen2009preconditioning,gergelits2019laplacian} for the potential idea of preconditioning by inverting Laplacian with our proposed multigrid solvers.
It is also possible to apply our proposed SPAI smoothers to elliptic optimal control problem \cite{he2022smoothing} involving Laplacian.
The  MATLAB codes for implementing our proposed algorithms are publicly available online at the link:
\url{https://github.com/junliu2050/SPAI-MG-Laplacian}.
%\section*{Acknowledgments}

  \section*{Appendix A: a proof of Theorem \ref{Thm_SAI92D}}
  \newcommand{\thechapter}{A}
  \setcounter{section}{0}
  \setcounter{theorem}{0}
  \renewcommand{\thetheorem}{\thechapter.\arabic{theorem}}
  \renewcommand{\thelemma}{\thechapter.\arabic{lemma}}
  \renewcommand{\theequation}{\thechapter.\arabic{equation}}
  In this appendix, we provide a detail proof of Theorem \ref{Thm_SAI92D} through several technical lemmas and propositions.

  If $(\th_1,\th_2)\in T^{\rm{H}}=\left[-\frac{\pi}{2}, \frac{3\pi}{2}\right)^2 \setminus \left[-\frac{\pi}{2}, \frac{\pi}{2}\right)^2$, then
  $(x_1,x_2):=(\cos\th_1,\cos\th_2)\in X^{\rm{H}}:=[-1,1]^2\setminus(0,1]^2.$
  Given $(a,b)\in\R^2$, we define
  \begin{equation}
  	f_{a,b}(x_1,x_2):=[b+a(x_1+x_2)+x_1x_2](2-x_1-x_2).
  \end{equation}
{Then,
 \[\mu_{\rm loc}= \max_{(x_1,x_2)\in X^{\rm{H}}}|1-\om f_{a,b}(x_1,x_2)|.\]
  By symmetry, we may assume without loss of generality that $x_1\le x_2$. Hence,
  \begin{equation}\label{Jab}
  	\mu_{\rm opt}= \min_{a,b,\omega}\max_{(x_1, x_2)\in X_1\cup X_2}|1-\om f_{a,b}(x_1,x_2)|,
  \end{equation} }
  where
  \begin{align}
  	X_1:=&\{(x_1,x_2)\in\R^2:~-1\le x_1\le0,~-x_1\le x_2\le1\},\\
  	X_2:=&\{(x_1,x_2)\in\R^2:~-1\le x_1\le0,~x_1\le x_2\le-x_1\}.
  \end{align}
  Since $X_1\cup X_2$ is compact, the extremes of $f_{a,b}$ on $X_1\cup X_2$ can be achieved.
  We denote
 \[\chi_{a,b}:=\max_{(x_1,x_2)\in X_1\cup X_2}f_{a,b}(x_1,x_2),~~m_{a,b}:=\min_{(x_1,x_2)\in X_1\cup X_2}f_{a,b}(x_1,x_2).\]
 Since $\widetilde{A}_h$ is positive, to guarantee the relaxation scheme convergent, we have to restrict $(a,b)$ in the following region
 \[R:=\{(a,b)\in\R^2:m_{a,b}>0\}.\] 
  To find a SPAI smoother achieving the optimal smoothing factor, from \eqref{mu_opt}, we need to solve the following minimization problem
  \begin{equation}
  	\mu_{\opt}=\min_{(a,b)\in R} J(a,b),
  \end{equation}
where 
 \[J(a,b)=(\chi_{a,b}-m_{a,b})/(\chi_{a,b}+m_{a,b}),\]  
  with the corresponding optimal choice of $\om$ in the optimization problem \eqref{Jab} is \[\om_{\opt}=2/(\chi_{a,b}+m_{a,b}).\]
  Moreover, by choosing $(x_1,x_2)$ to be $(-1,-1),(-1,0),(-1,1)$, and $(0,1)$, respectively, we obtain from the assumption $f_{a,b}(x_1,x_2)\ge m_{a,b}>0$ that
  \begin{equation}\label{ab-ine}
  	b>1,~~b>|a|,~~b-2a+1>0,
  \end{equation}
  provided $(a,b)\in R$.
  For $(x_1,x_2)\in X_1$, we have
 $x_1+x_2-1\le x_1x_2\le0$
  and hence,
  \[u_{a,b}(x_1+x_2)\le f_{a,b}(x_1,x_2)\le U_{a,b}(x_1+x_2),\]
  where
 \[u_{a,b}(t):=(b+at+t-1)(2-t),~~U_{a,b}(t):=(b+at)(2-t),\]
  with $t=x_1+x_2\in[0,1]$.
  Since $b>1$ and $b>|a|$ by \eqref{ab-ine}, we obtain
  \begin{align*}
  	\max_{(x_1,x_2)\in X_1}f_{a,b}(x_1,x_2)&=\max_{t\in[0,1]}U_{a,b}(t)=\begin{cases}
  		2b,&~b\ge 2a,\\
  		(b+2a)^2/(4a),&~b\le 2a,
  	\end{cases}\\
  	\min_{(x_1,x_2)\in X_1}f_{a,b}(x_1,x_2)&=\min_{t\in[0,1]}u_{a,b}(t)=\min\{b+a,2b-2\}.
  \end{align*}
  For $(x_1,x_2)\in X_2$, we have
  $-(x_1+x_2)-1\le x_1x_2\le(x_1+x_2)^2/4,$
  and hence,
  $v_{a,b}(-x_1-x_2)\le f_{a,b}(x_1,x_2)\le V_{a,b}(-x_1-x_2),$
  where
  \begin{equation}
  	v_{a,b}(t):=(b-at+t-1)(2+t),~~V_{a,b}(t):=(b-at+t^2/4)(2+t),
  \end{equation}
  with $t=-x_1-x_2\in[0,2]$.
  Recall from \eqref{ab-ine} that $b>1$ and $b-2a+1>0$. It is easily seen that $v_{a,b}(t)$ is monotone for $t\in[0,2]$, and hence
  \[\min_{(x_1,x_2)\in X_2}f_{a,b}(x_1,x_2)=\min_{t\in[0,2]}v_{a,b}(t)=\min\{4(b-2a+1),2b-2\}.\]
  To find the maximum of $V_{a,b}(t)$ with $t\in [0,2]$, we shall investigate its derivative
 \[V_{a,b}'(t)=(-a+t/2)(2+t)+b-at+t^2/4=3t^2/4+(1-2a)t+b-2a.\]
  If $(1-2a)^2<3(b-2a)$, then $V_{a,b}(t)$ is an increasing function for $t\in[0,2]$, and
  \[\max_{(x_1,x_2)\in X_2}f_{a,b}(x_1,x_2)=V_{a,b}(1)=4(b-2a+1).\]
  If $(1-2a)^2\ge3(b-2a)$, then solving $V_{a,b}'(t)=0$ gives two solutions
  \begin{equation}\label{tpm}
  	t_\pm={2a-1\pm\sqrt{(1-2a)^2-3(b-2a)}\over 3/2}.
  \end{equation}
  If $t_-\notin[0,2]$, then the maximum of $V_{a,b}(t)$ on  $[0,2]$ is achieved at the end point, and
  \[\max_{(x_1,x_2)\in X_2}f_{a,b}(x_1,x_2)=\max_{t\in[0,2]}V_{a,b}(t)=\max\{2b,4(b-2a+1)\}.\]
  If $t_-\in[0,2]$, then the maximum of $V_{a,b}(t)$ on  $[0,2]$ is achieved at $t_-$, and
  \[\max_{(x_1,x_2)\in X_2}f_{a,b}(x_1,x_2)=V_{a,b}(t_-)\ge V_{a,b}(0)=2b.\]
  A combination of the above arguments gives the following lemma.
  \begin{lemma}
  	If $(a,b)\in R$, then $b>1,~~b>|a|,~~b-2a+1>0.$
  	Moreover,
  	\[m_{a,b}=\min\{b+a,2b-2,4(b-2a+1)\},\]
  	and
  	\[\chi_{a,b}=\begin{cases}
  		\max\{\frac{(b+2a)^2}{4a},4(b-2a+1)\},&~b\le2a,\\
  		\max\{2b,4(b-2a+1)\},&~(1-2a)^2<3(b-2a),\\
  		\max\{2b,4(b-2a+1)\},&~(1-2a)^2\ge3(b-2a)>0,~t_-\notin[0,2],\\
  		V_{a,b}(t_-),&~(1-2a)^2\ge3(b-2a)>0,~t_-\in[0,2].
  	\end{cases} \]
  \end{lemma}
  For convenience of discussion, we divide $R$ into four disjoint sub-regions:
  \begin{align*}
  	R_0&:=\{(a,b)\in R,~b\le2a\},\\
  	R_1&:=\{(a,b)\in R,~(1-2a)^2<3(b-2a)\},\\
  	R_2&:=\{(a,b)\in R,~(1-2a)^2\ge3(b-2a)>0,~t_-\notin[0,2]\},\\
  	R_3&:=\{(a,b)\in R,~(1-2a)^2\ge3(b-2a)>0,~t_-\in[0,2]\}.
  \end{align*}
  Recall that
  \begin{equation}
  	J(a,b)={\chi_{a,b}-m_{a,b}\over \chi_{a,b}+m_{a,b}}={1-r_{a,b}\over 1+r_{a,b}},~~r_{a,b}:={m_{a,b}\over \chi_{a,b}}\in(0,1].
  \end{equation}
  We shall find lower bounds of $J(a,b)$ in the regions: $R_0$, $R_1\cup R_2$, and $R_3$, respectively.
  \begin{proposition}\label{prop-0}
  	If $(a,b)\in R_0$, then $J(a,b)\ge1/5$.
  \end{proposition}
  \begin{proof}
  	If $(a,b)\in R_0$, then $m_{a,b}=\min\{b+a,2b-2,4(b-2a+1)\},$ and $\chi_{a,b}=\max\{(b+2a)^2/(4a),4(b-2a+1)\}\ge(b+2a)^2/(4a)$.
  	There are three cases:
  	\begin{enumerate}[{Case} 1:]
  		\item   		$m_{a,b}=b+a$ and $\chi_{a,b}\ge(b+2a)^2/(4a)$. It can be shown that
  		\[ b\ge a+2,~b\ge3a-4/3,~b\le2a,~2\le a\le4/3,\]
  		a contradiction. Thus, this case does not exist.
  		\item   		$m_{a,b}=2b-2$ and $\chi_{a,b}\ge(b+2a)^2/(4a)$. It can be shown that
  	\[b\le a+2,~b\ge4a-3,~1<b\le2a,~1/2<a\le3/2,\]
  		from which we further obtain
  		\[r_{a,b}\le{4a(2b-2)\over (b+2a)^2}\le{8a(2a-1)\over(4a)^2}={2a-1\over2a}\le{2\over3}.\]
  		\item   		$m_{a,b}=4(b-2a+1)$ and $\chi_{a,b}=(b+2a)^2/(4a)$. It can be shown that
  	\[b\le 3a-4/3,~2a-1<b\le4a-3,~b\le2a,~a>1.\]
  		If $a\ge3/2$, then we obtain from $b\le2a$ that
  		\[r_{a,b}\le{16a(b-2a+1)\over (b+2a)^2}\le{16a\over(4a)^2}={1\over a}\le{2\over3}.\]
  		If $a\le3/2$, then we obtain from $b\le4a-3$ that
  		\[r_{a,b}\le{16a(b-2a+1)\over (b+2a)^2}\le{16a(2a-2)\over(6a-3)^2}\le{24\over6^2}={2\over3}.\]
  	\end{enumerate}
  	A combination of the above arguments yields $r_{a,b}\le2/3$ and hence
  	$J(a,b)\ge1/5$.
  \end{proof}

  \begin{proposition}\label{prop-12}
  	If $(a,b)\in R_1\cup R_2$, then $J(a,b)\ge3/17$.
  \end{proposition}
  \begin{proof}
  	If $(a,b)\in R_1\cup R_2$, then $m_{a,b}=\min\{b+a,2b-2,4(b-2a+1)\},$ and $\chi_{a,b}=\max\{2b,4(b-2a+1)\}$.
  	There are six cases to be considered.
  	\begin{enumerate}[{Case} 1:]
  		\item
  		$m_{a,b}=b+a$ and $\chi_{a,b}=2b$. It can be shown that
  		\[b\ge a+2,~b\ge3a-4/3,~b\le4a-2,~a\ge4/3,~b\ge2a,\]
  		from which we further obtain
  		$(1-2a)^2\ge3(b-2a)$ and $t_-\ge0.$
  		Since $(a,b)\in R_1\cup R_2$, we have $t_->2$, which implies
  		$a>2$, $b>6a-5$,
  		and
  	\[r_{a,b}={1\over2}+{a\over2b}<{1\over2}+{a\over2(6a-5)}<{1\over2}+{1\over7}={9\over14}.\]
  		\item
  		$m_{a,b}=b+a$ and $\chi_{a,b}=4(b-2a+1)$. It can be shown that
  	\[b\ge a+2,~b\ge3a-4/3,~b\ge4a-2.\]
  		If $a\le1/3$, then $b+a\le b-2a+2$, and
  	\[r_{a,b}={b+a\over4(b-2a+1)}\le{1\over4}.\]
  		If $a\ge4/3$, then we obtain from $b\ge4a-2$ that
  	\[r_{a,b}={b+a\over4(b-2a+1)}\le{5a-2\over4(2a-1)}\le{14/3\over20/3}={7\over10}.\]
  		If $1/3\le a\le4/3$, then we obtain from $b\ge a+2$ that
  		\[r_{a,b}={b+a\over4(b-2a+1)}\le{2a+2\over4(3-a)}\le{14/3\over20/3}={7\over10}.\]
  		\item
  		$m_{a,b}=2b-2$ and $\chi_{a,b}=2b$. It can be shown that
  		\[b\le a+2,~b\ge4a-3,~1<b\le4a-2,~3/4<a\le5/3,\]
  		from which we further obtain
  		$(1-2a)^2\ge3(b-2a)$ and $t_-\le2.$
  		Since $(a,b)\in R_1\cup R_2$, we have $t_-<0$, which implies $b<2a\le10/3$ and
  	\[r_{a,b}=1-{1\over b}<{7\over10}.\]
  		\item
  		$m_{a,b}=2b-2$ and $\chi_{a,b}=4(b-2a+1)$. It can be shown that
  	\[b\le a+2,~b\ge4a-3,~b\ge4a-2,~a\le4/3.\]
  		If $a\le1$, then $b-1\le b-2a+1$ and
  $r_{a,b}={b-1\over2(b-2a+1)}\le{1\over2}.$
  		If $a\ge1$, then we obtain from $b\ge a+2$ and $a\le4/3$ that
  	\[r_{a,b}={b-1\over2(b-2a+1)}\le{a+1\over2(3-a)}\le{7/3\over10/3}={7\over10}.\]
  		\item
  		$m_{a,b}=4(b-2a+1)$ and $\chi_{a,b}=2b$. It can be shown that
  	\[b\le 3a-4/3,~b\le4a-3,~b\le4a-2,\]
  		from which we further obtain
  		$(1-2a)^2\ge3(b-2a).$
  		It also follows from $b>1$ that $a>1$.
  		Since $(a,b)\in R_1\cup R_2$, we have either $t_-<0$ or $t_->2$.
  		If $t_->2$, then we have $a>2$ and $b>6a-5$, which contradicts $b\le4a-3$.
  		Hence, we have $t_-<0$, and consequently, $b<2a$. If $a\ge3/2$, then we obtain from $b<2a$ that
  		\[r_{a,b}={2(b-2a+1)\over b}\le{1\over a}\le{2\over3}.\]
  		If $a\le3/2$, then we obtain from $b\le4a-3$ that
  		\[r_{a,b}={2(b-2a+1)\over b}\le{2(2a-2)\over 4a-3}\le{2\over3}.\]
  		\item
  		$m_{a,b}=4(b-2a+1)$ and $\chi_{a,b}=4(b-2a+1)$. It can be shown that
  		\[b\le 3a-4/3,~b\le4a-3,~b\ge4a-2,\]
  		a contradiction. Hence, this case does not exist.
  	\end{enumerate}
  	A combination of all cases gives $r_{a,b}\le7/10$ and $J(a,b)\ge3/17$ for all $(a,b)\in R_1\cup R_2$.
  \end{proof}
  It remains to find the lower bound of $J(a,b)$ in $R_3$.
  If $(a,b)\in R_3$, then $3b\le 4a^2+2a+1$ and $0\le t_-\le2$.
  From the quadrature formula of $t_-$ in \eqref{tpm}, we obtain $a\ge(2+3t_-)/4$, $b\ge 2a$ and
  \[2a-4\le\sqrt{4a^2+2a+1-3b}.\]
  If $a\ge2$, then $b\le6a-5$. It follows from $\chi_{a,b}\ge f_{a,b}(0,0)=2b$ and $m_{a,b}\le f_{a,b}(0,1)=b+a$ that
  \[r_{a,b}={m_{a,b}\over \chi_{a,b}}\le{1\over2}+{a\over2b}\le{1\over2}+{a\over2(6a-5)}<{1\over2}+{1\over7}={9\over14}.\]
  Now, we assume $a\le2$.
  Since $V_{a,b}'(t_-)=0$, we have $b=-3t_-^2/4+(2a-1)t_-+2a$.
  Recall that $m_{a,b}=\min\{b+a,2b-2,4(b-2a+1)\},$ and $\chi_{a,b}=V_{a,b}(t_-)\ge\max\{2b,4(b-2a+1)\}$.
  We obtain
  \[r_{a,b}={m_{a,b}\over \chi_{a,b}}={w(t_-,a)\over W(t_-,a)},\]
  where $w(t_-,a)=\min\{w_1(t_-,a),w_2(t_-,a),w_3(t_-,a)\}$, and
  \begin{align*}
  	w_1(t_-,a)&=b+a=-3t_-^2/4+(2a-1)t_-+3a=(2t_-+3)a-(3t_-^2/4+t_-),\\
  	w_2(t_-,a)&=2b-2=-3t_-^2/2+2(2a-1)t_-+4a-2\\
  	&=(4t_-+4)a-(3t_-^2/2+2t_-+2),\\
  	w_3(t_-,a)&=4(b-2a+1)=-3t_-^2+4(2a-1)t_-+4=8t_-a-(3t_-^2+4t_--4),\\
  	W(t_-,a)&=V_{a,b}(t_-)=[-t_-^2/2+(a-1)t_-+2a](2+t_-)\nonumber\\
  	&=[(t_-+2)a-(t_-^2/2+t_-)](t_-+2).
  \end{align*}
  Replacing $t_-$ with $t$ in the above formulas,
  we obtain the bi-variate polynomials $w_1(t,a),~w_2(t,a),~w_3(t,a)$, and $W(t,a)$.
  Denote
  \begin{equation}
  	r^*:=\max_{(t,a)\in S}{w(t,a)\over W(t,a)},
  \end{equation}
  where
 \[
  	S:=\{(t,a)\in\R^2:0\le t\le2,~(3t+2)/4\le a\le2\}.
\]
  The above arguments can be summarized in the following lemma.
  \begin{lemma}
  	For any $(a,b)\in R_3$ with $a\ge 2$, we have $r_{a,b}\le9/14$.
  	For any $(a,b)\in R_3$ with $a\le 2$, we have $r_{a,b}\le r^*$.
  \end{lemma}
  Next, we will estimate $r^*$ in a sequence of lemmas.
  \begin{lemma}
  	Fix $t\in[0,2]$. The functions $w_1(t,a),w_2(t,a),w_3(t,a),W(t,a)$ are linear, monotonically increasing, and positive for $(3t+2)/4<a\le 2$. The functions $w_1(t,a),w_3(t,a),W(t,a)$ are positive and $w_2(t,a)\ge0$ at $a=(3t+2)/4$. Moreover,
  	\begin{enumerate}
  		\item  If $t\in(0,2)$, then the rational functions $w_1(t,a)/W(t,a)$, $w_2(t,a)/W(t,a)$, and $w_3(t,a)/W(t,a)$, respectively, are monotonically decreasing, increasing, and decreasing for $a\in[(3t+2)/4,2]$.
  		\item  If $t=0$, then {$w_1(t,a)/W(t,a)$} is a constant function, $w_2(t,a)/W(t,a)$ is an increasing function, and $w_3(t,a)/W(t,a)$ is a decreasing function for $a\in[(3t+2)/4,2]$.
  		\item  If $t=2$, then {the interval $[(3t+2)/4,2]$ shrinks to a point $2$}.
  	\end{enumerate}
  \end{lemma}
  \begin{proof}
  	It is obvious that the linear functions $w_1(t,a),w_2(t,a),w_3(t,a),W(t,a)$ have positive leading coefficients and are monotonically increasing. We verify that
  	\begin{align*}
  		w_1(t,(3t+2)/4)&=(2t+3)(3t+2)/4-(3t^2/4+t)=3t^2/4+9t/4+3/2>0,\\
  		w_2(t,(3t+2)/4)&=(4t+4)(3t+2)/4-(3t^2/2+2t+2)=3t^2/2+3t\ge0,\\
  		w_3(t,(3t+2)/4)&=8t(3t+2)/4-(3t^2+4t-4)=3t^2+4>0,\\
  		W(t,(3t+2)/4)&=[(t+2)(3t+2)/4-(t^2/2+t)](t+2)=(t^2/4+t+1)(t+2)>0.
  	\end{align*}
  	This proves that $w_1(t,a),w_3(t,a),W(t,a)$ are positive at $a=(3t+2)/4$, and $w_2(t,a)\ge 0$ at $a=(3t+2)/4$. Consequently, all functions $w_1(t,a),w_2(t,a),w_3(t,a),W(t,a)$ are positive for $(3t+2)/4<a\le 2$.
  	
  	For any $k_1>0$ and $k_2>0$, the M{\"o}bius transform \cite{arnold2008mobius}
  	\[{k_1x-b_1\over k_2x-b_2}={k_1\over k_2}+{k_1b_2-k_2b_1\over k_2(k_2x-b_2)}\]
  	is an increasing (resp. decreasing) function for $x>b_2/k_2$ if $k_1b_2-k_2b_1$ is negative (resp. positive).
  	Given $t\in(0,2)$, it then follows from
  	\begin{align*}
  		(2t+3)(t^2/2+t)-(t+2)(3t^2/4+t)&=t^3/4+t^2+t>0,\\
  		(4t+4)(\frac{1}{2}t^2+t)-(t+2)(\frac{3}{2}t^2+2t+2)&=\frac{1}{2}t^3+t^2-2t-4=\frac{1}{2}(t+2)^2(t-2)<0,\\
  		8t(t^2/2+t)-(t+2)(3t^2+4t-4)&=t^3-2t^2-4t+8=(t-2)^2(t+2)>0
  	\end{align*}
  	that the rational functions $w_1(t,a)/W(t,a)$, $w_2(t,a)/W(t,a)$, and $w_3(t,a)/W(t,a)$, respectively, are monotonically decreasing, increasing, and decreasing for $a\in[(3t+2)/4,2]$. The proof for the cases $t=0$ and $t=2$ are trivial.
  \end{proof}

  \begin{lemma}
  	If $t\in[1,2]$, then $w(t,a)=w_1(t,a)$ and $w(t,a)/W(t,a)\le2/3$ for all $a\in[(3t+2)/4,2]$.
  \end{lemma}
  \begin{proof}
  	For any $a\in[(3t+2)/4,2]$, since
  	\begin{align*}
  		w_2(t,a)-w_1(t,a)&=(2t+1)a-(3t^2/4+t+2)
  		\\&\ge (2t+1)(3t+2)/4-(3t^2/4+t+2)
  		 =3(t-1)(t+2)/4
  		 \ge0,
  	\end{align*}
  	and
  	\begin{align*}
  		w_3(t,a)-w_1(t,a)&=(6t-3)a-(9t^2/4+3t-4)
  		\\&\ge(6t-3)(3t+2)/4-(9t^2/4+3t-4)
  		 =9t^2/4-9t/4+5/2
  		 >0,
  	\end{align*}
  	we have $w(t,a)=\min\{w_1(t,a),w_2(t,a),w_3(t,a)\}=w_1(t,a)$.
  	Consequently,
  \[{w(t,a)\over W(t,a)}\le {w_1(t,(3t+2)/4)\over W(t,(3t+2)/4)}={3(t+1)\over(t+2)^2}\le{2\over3}.\]
  	This completes the proof.
  \end{proof}

  \begin{lemma}
  	If $t\in[1/2,1]$, then $w_1(t,a)<w_3(t,a)$ and $w(t,a)/W(t,a)\le24/35$ for all $a\in[(3t+2)/4,2]$.
  \end{lemma}
  \begin{proof}
  	It is easily seen that
  	\begin{align*}
  		w_3(t,a)-w_1(t,a)&=(6t-3)a-(9t^2/4+3t-4)
  		\\&\ge(6t-3)(3t+2)/4-(9t^2/4+3t-4)
  		 =9t^2/4-9t/4+5/2
  		 >0.
  	\end{align*}
  	We also note that
  	\begin{align*}
  		w_2(t,(3t+2)/4)-w_1(t,(3t+2)/4)&=3(t-1)(t+2)/2\le0.
  	\end{align*}
  	If $a\ge (3t^2/4+t+2)/(2t+1)$, then
  	\[w_2(t,a)-w_1(t,a)=(2t+1)a-(3t^2/4+t+2)\ge0,\]
  	which implies $w(t,a)=w_1(t,a)$, and hence
  	\begin{align*}
  		{w(t,a)\over W(t,a)}&\le {w_1(t,(3t^2/4+t+2)/(2t+1))\over W(t,(3t^2/4+t+2)/(2t+1))}
  		={6(t+2)\over-t^3+12t+16}
  		\le{24\over35}.
  	\end{align*}
  	If $a\le (3t^2/4+t+2)/(2t+1)$, then
  \[w_2(t,a)-w_1(t,a)=(2t+1)a-(3t^2/4+t+2)\le0,\]
  	which implies $w(t,a)=w_2(t,a)$ and
  	\begin{align*}
  		{w(t,a)\over W(t,a)}&\le {w_2(t,(3t^2/4+t+2)/(2t+1))\over W(t,(3t^2/4+t+2)/(2t+1))}
  		={6(t+2)\over-t^3+12t+16}
  		\le{24\over35}.
  	\end{align*}
  	The proof is completed.
  \end{proof}

  \begin{lemma}
  	If $t\in[0,1/2]$, then $w(t,a)/W(t,a)\le(68-5\sqrt{10})/72$ for all $a\in[(3t+2)/4,2]$.
  \end{lemma}
  \begin{proof}
  	We have to consider three cases.
  	\begin{enumerate}[{Case} 1:]
  		\item $w(t,a)=w_1(t,a)$.
  		It follows from
  		\begin{align*}
  			0\le w_2(t,a)-w_1(t,a)=(2t+1)a-(3t^2/4+t+2),
  		\end{align*}
  		and
  		\begin{align*}
  			0\ge w_1(t,a)-w_3(t,a)=(3-6t)a-(4-3t-9t^2/4)
  		\end{align*}
  		that
  		\begin{align*}
  			{3t^2/4+t+2\over2t+1}\le a\le{4-3t-9t^2/4\over3-6t},
  		\end{align*}
  		which together with $t\in[0,2]$ gives
  	$t\ge t_0:={14-4\sqrt{10}\over9}.$
  		Consequently,
  		\begin{align*}
  			{w(t,a)\over W(t,a)}&\le {w_1(t,(3t^2/4+t+2)/(2t+1))\over W(t,(3t^2/4+t+2)/(2t+1))}
  			={6(t+2)\over-t^3+12t+16}
  			\le{68-5\sqrt{10}\over72}.
  		\end{align*}
  		
  		\item $w(t,a)=w_2(t,a)$.
  		It follows from
  		\begin{align*}
  			0\ge w_2(t,a)-w_1(t,a)=(2t+1)a-(3t^2/4+t+2),
  		\end{align*}
  		and
  		\begin{align*}
  			0\ge w_2(t,a)-w_3(t,a)=(4-4t)a-(6-2t-3t^2/2)
  		\end{align*}
  		that
  		\begin{align*}
  			a\le\min\{{3t^2/4+t+2\over2t+1},{6-2t-3t^2/2\over4-4t}\}.
  		\end{align*}
  		If
  		\[{3t^2/4+t+2\over2t+1}\le{6-2t-3t^2/2\over4-4t},\]
  		then
  $t\ge t_0:={14-4\sqrt{10}\over9},$
  		and
  		\begin{align*}
  			{w(t,a)\over W(t,a)}&\le {w_2(t,(3t^2/4+t+2)/(2t+1))\over W(t,(3t^2/4+t+2)/(2t+1))}
  			={6(t+2)\over-t^3+12t+16}
  			\le{68-5\sqrt{10}\over72}.
  		\end{align*}
  		If
  		\[{3t^2/4+t+2\over2t+1}\ge{6-2t-3t^2/2\over4-4t},\]
  		then
  	$t\le t_0:={14-4\sqrt{10}\over9},$
  		and
  		\begin{align*}
  			{w(t,a)\over W(t,a)}&\le {w_2(t,(6-2t-3t^2/2)/(4-4t))\over W(t,(6-2t-3t^2/2)/(4-4t))} \\
  			&={8(2+3t)\over(12+4t-t^2)(2+t)}
  			\le{68-5\sqrt{10}\over72}.
  		\end{align*}
  		
  		\item $w(t,a)=w_3(t,a)$.
  		It follows from
  		\begin{align*}
  			0\le w_1(t,a)-w_3(t,a)=(3-6t)a-(4-3t-9t^2/4)
  		\end{align*}
  		and
  		\begin{align*}
  			0\le w_2(t,a)-w_3(t,a)=(4-4t)a-(6-2t-3t^2/2)
  		\end{align*}
  		that
  		\begin{align*}
  			a\ge\max\{{4-3t-9t^2/4\over3-6t},{6-2t-3t^2/2\over4-4t}\}.
  		\end{align*}
  		If
  		\[{4-3t-9t^2/4\over3-6t}\ge{6-2t-3t^2/2\over4-4t},\]
  		then
  		$t\ge t_0:={14-4\sqrt{10}\over9},$
  		and
  		\begin{align*}
  			{w(t,a)\over W(t,a)}&\le {w_3(t,(4-3t-9t^2/4)/(3-6t))\over W(t,(4-3t-9t^2/4)/(3-6t))} \\
  			&={4(-9t^2-4t+12)\over(3t^2-18t+16)(t+2)^2}
  			\le{68-5\sqrt{10}\over72}.
  		\end{align*}
  		If
  		\[{4-3t-9t^2/4\over3-6t}\le{6-2t-3t^2/2\over4-4t},\]
  		then
  		$t\le t_0:={14-4\sqrt{10}\over9},$
  		and
  		\begin{align*}
  			{w(t,a)\over W(t,a)}&\le {w_3(t,(6-2t-3t^2/2)/(4-4t))\over W(t,(6-2t-3t^2/2)/(4-4t))}\\
  			&={8(2+3t)\over(12+4t-t^2)(2+t)}
  			\le{68-5\sqrt{10}\over72}.
  		\end{align*}
  	\end{enumerate}
  	This completes the proof.
  \end{proof}
  A combination of the above lemmas gives $r_{a,b}\le r^*\le(68-5\sqrt{10})/72$ for $(a,b)\in R_3$.
  Since $J(a,b)=(1-r_{a,b})/(1+r_{a,b})$, we have the following result.
  \begin{proposition}\label{prop-3}
  	If $(a,b)\in R_3$, then $J(a,b)\ge(9+8\sqrt{10})/215$.
  \end{proposition}

\begin{proof}[Proof of Theorem \ref{Thm_SAI92D}]
  It follows from Propositions \ref{prop-0}, \ref{prop-12}, and  \ref{prop-3} that
  $$J(a,b)\ge(9+8\sqrt{10})/215$$
  for all $(a,b)\in R=R_0\cup R_1\cup R_2\cup R_3$.
  On the other hand, by choosing $(a,b)=(5/3,11/3)\in R_3$, we obtain from a simple calculation that \[J(5/3,11/3)=(9+8\sqrt{10})/215\]
   with $m_{5/3,11/3}=16/3$ and $\chi_{5/3,11/3}=(4352+320\sqrt{10})/729$.
  Hence, we have
 \[\mu_{\opt}=\min_{(a,b)\in R} J(a,b)={9+8\sqrt{10}\over215},\]
  and the corresponding optimal value of $\om$ is calculated as
  \[\om_{\opt}={2\over m_{5/3,11/3}+\chi_{5/3,11/3}}={309-12\sqrt{10}\over1720},\]
  which completes the proof of Theorem \ref{Thm_SAI92D}.
    \end{proof}
\bibliographystyle{siamplain}
\bibliography{Control_Laplacebib,EllipticControl,../LFA,../waveControl,../waveControl2}

\begin{thebibliography}{10}

\bibitem{adams2003parallel}
{\sc M.~Adams, M.~Brezina, J.~Hu, and R.~Tuminaro}, {\em Parallel multigrid
  smoothing: polynomial versus {G}auss--{S}eidel}, Journal of Computational
  Physics, 188 (2003), pp.~593--610.

\bibitem{arnold2008mobius}
{\sc D.~N. Arnold and J.~P. Rogness}, {\em M{\"o}bius transformations
  revealed}, Notices of the American Mathematical Society, 55 (2008),
  pp.~1226--1231.

\bibitem{baker2011multigrid}
{\sc A.~H. Baker, R.~D. Falgout, T.~V. Kolev, and U.~M. Yang}, {\em Multigrid
  smoothers for ultraparallel computing}, SIAM Journal on Scientific Computing,
  33 (2011), pp.~2864--2887.

\bibitem{benzi1996sparse}
{\sc M.~Benzi, C.~D. Meyer, and M.~T{\^{u}}ma}, {\em A sparse approximate
  inverse preconditioner for the conjugate gradient method}, SIAM Journal on
  Scientific Computing, 17 (1996), pp.~1135--1149.

\bibitem{benzi1998sparse}
{\sc M.~Benzi and M.~Tuma}, {\em A sparse approximate inverse preconditioner
  for nonsymmetric linear systems}, SIAM Journal on Scientific Computing, 19
  (1998), pp.~968--994.

\bibitem{Benzi1999}
{\sc M.~Benzi and M.~T{\^{u}}ma}, {\em A comparative study of sparse
  approximate inverse preconditioners}, Applied Numerical Mathematics, 30
  (1999), pp.~305--340.

\bibitem{bertaccini2018iterative}
{\sc D.~Bertaccini and F.~Durastante}, {\em Iterative methods and
  preconditioning for large and sparse linear systems with applications},
  Chapman and Hall/CRC, 2018.

\bibitem{bertaccini2016sparse}
{\sc D.~Bertaccini and S.~Filippone}, {\em Sparse approximate inverse
  preconditioners on high performance {GPU} platforms}, Computers \&
  Mathematics with Applications, 71 (2016), pp.~693--711.

\bibitem{birken2011designing}
{\sc P.~Birken}, {\em Designing optimal smoothers for multigrid methods for
  unsteady flow problems}, in 20th AIAA Computational Fluid Dynamics
  Conference, 2011, p.~3233.

\bibitem{Bolten2016}
{\sc M.~Bolten, T.~K. Huckle, and C.~D. Kravvaritis}, {\em Sparse matrix
  approximations for multigrid methods}, Linear Algebra and its Applications,
  502 (2016), pp.~58--76.

\bibitem{braess1981contraction}
{\sc D.~Braess}, {\em The contraction number of a multigrid method for solving
  the {P}oisson equation}, Numerische Mathematik, 37 (1981), pp.~387--404.

\bibitem{braess1982convergence}
{\sc D.~Braess}, {\em The convergence rate of a multigrid method with
  {G}auss-{S}eidel relaxation for the {P}oisson equation}, in Multigrid
  methods, Springer, 1982, pp.~368--386.

\bibitem{braess1984convergence}
{\sc D.~Braess}, {\em The convergence rate of a multigrid method with
  {G}auss-{S}eidel relaxation for the {P}oisson equation}, Mathematics of
  computation, 42 (1984), pp.~505--519.

\bibitem{branden2022grid}
{\sc H.~Brand{\'e}n}, {\em Grid independent convergence using multilevel
  circulant preconditioning: {P}oisson’s equation}, BIT Numerical
  Mathematics, 62 (2022), pp.~409--429.

\bibitem{Brandt1977}
{\sc A.~Brandt}, {\em Multi-level adaptive solutions to boundary-value
  problems}, Mathematics of Computation, 31 (1977), pp.~333--390.

\bibitem{Brannick2015}
{\sc J.~Brannick, X.~Hu, C.~Rodrigo, and L.~Zikatanov}, {\em Local {F}ourier
  analysis of multigrid methods with polynomial smoothers and aggressive
  coarsening}, Numerical Mathematics: Theory, Methods and Applications, 8
  (2015), pp.~1--21.

\bibitem{Briggs2000}
{\sc W.~L. Briggs, V.~E. Henson, and S.~F. McCormick}, {\em A multigrid
  tutorial}, SIAM, Philadelphia, PA, 2000.

\bibitem{Brker2002}
{\sc O.~Br\"{o}ker}, {\em Sparse approximate inverse smoothers for geometric
  and algebraic multigrid}, Applied Numerical Mathematics, 41 (2002),
  pp.~61--80.

\bibitem{Brker2001}
{\sc O.~Br\"{o}ker, M.~J. Grote, C.~Mayer, and A.~Reusken}, {\em Robust
  parallel smoothing for multigrid via sparse approximate inverses}, {SIAM}
  Journal on Scientific Computing, 23 (2001), pp.~1396--1417.

\bibitem{brown2021tuning}
{\sc J.~Brown, Y.~He, S.~MacLachlan, M.~Menickelly, and S.~M. Wild}, {\em
  Tuning multigrid methods with robust optimization and local fourier
  analysis}, SIAM Journal on Scientific Computing, 43 (2021), pp.~A109--A138.

\bibitem{cosgrove1992approximate}
{\sc J.~Cosgrove, J.~Diaz, and A.~Griewank}, {\em Approximate inverse
  preconditionings for sparse linear systems}, International Journal of
  Computer Mathematics, 44 (1992), pp.~91--110.

\bibitem{de2021two}
{\sc {\'A}.~P. de~la Riva, C.~Rodrigo, and F.~J. Gaspar}, {\em A two-level
  method for isogeometric discretizations based on multiplicative {S}chwarz
  iterations}, Computers \& Mathematics with Applications, 100 (2021),
  pp.~41--50.

\bibitem{gaspar2009fourier}
{\sc F.~J. Gaspar, J.~L. Gracia, and F.~J. Lisbona}, {\em Fourier analysis for
  multigrid methods on triangular grids}, SIAM Journal on Scientific Computing,
  31 (2009), pp.~2081--2102.

\bibitem{gaspar2009geometric}
{\sc F.~J. Gaspar, J.~L. Gracia, F.~J. Lisbona, and C.~Rodrigo}, {\em On
  geometric multigrid methods for triangular grids using three-coarsening
  strategy}, Applied Numerical Mathematics, 59 (2009), pp.~1693--1708.

\bibitem{gergelits2019laplacian}
{\sc T.~Gergelits, K.-A. Mardal, B.~F. Nielsen, and Z.~Strakos}, {\em Laplacian
  preconditioning of elliptic {PDE}s: Localization of the eigenvalues of the
  discretized operator}, SIAM Journal on Numerical Analysis, 57 (2019),
  pp.~1369--1394.

\bibitem{gould1998sparse}
{\sc N.~I. Gould and J.~A. Scott}, {\em Sparse approximate-inverse
  preconditioners using norm-minimization techniques}, SIAM Journal on
  Scientific Computing, 19 (1998), pp.~605--625.

\bibitem{CH2021addVanka}
{\sc C.~Greif and Y.~He}, {\em A closed-form multigrid smoothing factor for an
  additive {V}anka-type smoother applied to the {P}oisson equation}, arXiv
  preprint arXiv:2111.03190,  (2021).

\bibitem{Grote1997}
{\sc M.~J. Grote and T.~Huckle}, {\em Parallel preconditioning with sparse
  approximate inverses}, {SIAM} Journal on Scientific Computing, 18 (1997),
  pp.~838--853.

\bibitem{guillet2011simple}
{\sc T.~Guillet and R.~Teyssier}, {\em A simple multigrid scheme for solving
  the {P}oisson equation with arbitrary domain boundaries}, Journal of
  Computational Physics, 230 (2011), pp.~4756--4771.

\bibitem{he2022smoothing}
{\sc Y.~He and J.~Liu}, {\em Smoothing analysis of two robust multigrid methods
  for elliptic optimal control problems}, arXiv preprint arXiv:2203.13066,
  (2022).

\bibitem{he2020two}
{\sc Y.~He and S.~MacLachlan}, {\em Two-level {F}ourier analysis of multigrid
  for higher-order finite-element discretizations of the {L}aplacian},
  Numerical Linear Algebra with Applications, 27 (2020), p.~e2285.

\bibitem{Hocking2011}
{\sc L.~R. Hocking and C.~Greif}, {\em Closed-form multigrid smoothing factors
  for lexicographic {G}auss-{S}eidel}, {IMA} Journal of Numerical Analysis, 32
  (2011), pp.~795--812.

\bibitem{holter1986vectorized}
{\sc W.~H. Holter}, {\em A vectorized multigrid solver for the
  three-dimensional {P}oisson equation}, Applied mathematics and computation,
  19 (1986), pp.~127--144.

\bibitem{huang2021learning}
{\sc R.~Huang, R.~Li, and Y.~Xi}, {\em Learning optimal multigrid smoothers via
  neural networks}, arXiv preprint arXiv:2102.12071,  (2021).

\bibitem{Kraus2012}
{\sc J.~Kraus, P.~Vassilevski, and L.~Zikatanov}, {\em Polynomial of best
  uniform approximation to and smoothing in two-level methods}, Computational
  Methods in Applied Mathematics, 12 (2012), pp.~448--468.

\bibitem{lottes2022optimal}
{\sc J.~Lottes}, {\em Optimal polynomial smoothers for multigrid {V}-cycles},
  arXiv preprint arXiv:2202.08830,  (2022).

\bibitem{nielsen2009preconditioning}
{\sc B.~F. Nielsen, A.~Tveito, and W.~Hackbusch}, {\em Preconditioning by
  inverting the {L}aplacian: an analysis of the eigenvalues}, IMA Journal of
  Numerical Analysis, 29 (2009), pp.~24--42.

\bibitem{notay2015massively}
{\sc Y.~Notay and A.~Napov}, {\em A massively parallel solver for discrete
  {P}oisson-like problems}, Journal of Computational Physics, 281 (2015),
  pp.~237--250.

\bibitem{pan2017extrapolation}
{\sc K.~Pan, D.~He, and H.~Hu}, {\em An extrapolation cascadic multigrid method
  combined with a fourth-order compact scheme for {3D} poisson equation},
  Journal of Scientific Computing, 70 (2017), pp.~1180--1203.

\bibitem{rodrigo2017validity}
{\sc C.~Rodrigo}, {\em On the validity of the local {F}ourier analysis},
  Journal of Computational Mathematics, 37 (2019), pp.~340--348.

\bibitem{saad2003iterative}
{\sc Y.~Saad}, {\em Iterative Methods for Sparse Linear Systems: Second
  Edition}, SIAM, Philadelphia, PA, 2003.

\bibitem{tang1999toward}
{\sc W.-P. Tang}, {\em Toward an effective sparse approximate inverse
  preconditioner}, {SIAM} Journal on Matrix Analysis and Applications, 20
  (1999), pp.~970--986.

\bibitem{Tang2000}
{\sc W.-P. Tang and W.~L. Wan}, {\em Sparse approximate inverse smoother for
  multigrid}, {SIAM} Journal on Matrix Analysis and Applications, 21 (2000),
  pp.~1236--1252.

\bibitem{trottenberg2000multigrid}
{\sc U.~Trottenberg, C.~W. Oosterlee, and A.~Schuller}, {\em Multigrid},
  Academic press, 2000.

\bibitem{wienands2004practical}
{\sc R.~Wienands and W.~Joppich}, {\em Practical {F}ourier analysis for
  multigrid methods}, CRC press, 2004.

\bibitem{Yang2017}
{\sc X.~Yang and R.~Mittal}, {\em Efficient relaxed-{J}acobi smoothers for
  multigrid on parallel computers}, Journal of Computational Physics, 332
  (2017), pp.~135--142.

\bibitem{zhang1996acceleration}
{\sc J.~Zhang}, {\em Acceleration of five-point red-black {G}auss-{S}eidel in
  multigrid for {P}oisson equation}, Applied Mathematics and Computation, 80
  (1996), pp.~73--93.

\bibitem{zhang1998fast}
{\sc J.~Zhang}, {\em Fast and high accuracy multigrid solution of the three
  dimensional {P}oisson equation}, Journal of Computational Physics, 143
  (1998), pp.~449--461.

\end{thebibliography}
\end{document}